\documentclass[10pt]{article}

\usepackage{amsmath}
\usepackage{amsfonts}
\usepackage{amsthm}
\usepackage{amssymb}

\newtheorem{Theorem}{Theorem}[section]
\newtheorem{Definition}[Theorem]{Definition}
\newtheorem{Proposition}[Theorem]{Proposition}
\newtheorem{Lemma}[Theorem]{Lemma}

\newtheorem{Remark}[Theorem]{Remark}

\numberwithin{equation}{section}

\def\bar{\overline}
\def\real{\mathbb{R}}
\def\nat{\mathbb{N}}

\def\liminfn{\liminf\limits_{n\rightarrow\infty}}
\def\limsupn{\limsup\limits_{n\rightarrow\infty}}

\newcommand{\dual}[2]{\left\langle #1 \right\rangle_{{#2}^*\times #2}}

\begin{document}

\title{
Convergence of a Time Discretization for Nonlinear Second Order Inclusion\thanks{The research was
  supported by the Marie Curie International Research Staff
  Exchange Scheme Fellowship within the 7th European Community Framework Programme under Grant
  Agreement No. 295118,
  the International Project co-financed by the Ministry of Science and Higher Education of Republic of
  Poland under grant no. W111/7.PR/2012
  and the National Science Center of Poland under Maestro
  Advanced Project no. DEC-2012/06/A/ST1/00262.}}

\author{
 Krzysztof Bartosz$^{\,1}$, \ \
 Leszek Gasi\'nski$^{\,1}$\footnote{Corresponding author email: Leszek.Gasinski@ii.uj.edu.pl}, \ \
 Zhenhai Liu$^{\, 2}$ \ \ \\
 and Pawe{\l} Szafraniec$^{\,1}$, \\ ~ \\
{\small $^1$ Jagiellonian University, Faculty of Mathematics and Computer Science} \\
{\small ul. Lojasiewicza 6, 30348 Krakow, Poland} \\ ~ \\
{\small $^2$ College of Sciences, Guangxi University for Nationalities} \\
{\small Nanning, Guangxi 530006, Peoples Republic of China}
}

\maketitle

\thispagestyle{empty}

\

\noindent {\bf Abstract.}
  We study an abstract second order inclusion involving two nonlinear single-valued operators and a nonlinear multivalued term. Our goal is to
  establish the existence of solutions to the problem by applying numerical scheme based on time discretization. We show that the sequence of
  approximate solution converges weakly to a solution of the exact problem. We apply our abstract result to a dynamic, second order in time
  differential inclusion involving Clarke subdifferential of a locally Lipschitz, possibly nonconvex and nonsmooth potential. In two presented examples the Clarke
  subdifferential appears either in a source term or in a boundary term.

\vskip 4mm


\vskip 4mm


\newpage

%%%%%%%%%%%%%%%%%%%%%%%%%%%%%%%%%%%
\section{Introduction}\label{Introduction}
In this paper we study the following inclusion problem
\begin{equation}\label{KB_17}
\begin{cases}
  u''(t)+A(t,u'(t))+B(t,u(t))+\gamma^*M(\gamma u' (t))\ni f(t)\\
  u(0)=u_0,\ u'(0)=v_0
\end{cases}
\end{equation}
and we deal with the existence of its solution in an appropriate function space. In the above problem, $A$ and $B$ are single-valued, nonlinear operators, $M$
is a multivalued term, $\gamma$ is a linear, continuous and compact operator and $\gamma^*$ denotes its adjoint operator. Our goal is to generalize the result
obtained in \cite{Emrichh_2010}, where the second order equation has been studied. In our case, because of the presence of multivalued term $M$, we need to
apply a technique taken from set valued analysis. Moreover, in \cite{Emrichh_2010}, the operator $A$ is assumed to be hemicontinuous and monotone with respect
to the second variable. In comparison to \cite{Emrichh_2010}, we assume that it is only pseudomonotone which allows to deal with a larger class of operators.
On the other hand, it forces to use more advanced approach. Similarly as in \cite{Emrichh_2010}, the operator $B$ is assumed to be a nonlinear perturbation of
a linear principal part $B_0$. Using the idea presented in \cite{Emrichh_2010}, we start with the following numerical scheme
\begin{eqnarray}\label{KB_18}
  &\frac{2}{\tau_{n+1}+\tau_n}\bigg(\frac{u^{n+1}-u^n}{\tau_{n+1}}-\frac{u^{n}-u^{n-1}}{\tau_{n}}\bigg)+A\bigg(t_n,\frac{u^{n+1}-u^n}{\tau_{n+1}}\bigg)\cr
  &+B(t_n,u^n)+\gamma^* M\bigg(\gamma\frac{u^{n+1}-u^n}{\tau_{n+1}}\bigg)\ni f^n,\quad n=1,...,N-1,
\end{eqnarray}
with initial condition. We obtain a solution $\{u^n\}$ applying an existence result for a corresponding elliptic inclusion in each fixed time step. To this
end, we use a surjectivity result for a pseudomonotone, coercive multivalued operator. Having the solution of the time-semidiscrete  problem (\ref{KB_18}), we
construct a sequence $u_\tau$ of piecewise constant functions in order to approximate a solution of \eqref{KB_17} and sequences $v_\tau$ and $\hat{v}_\tau$ of piecewise constant and piecewise linear functions  in order to approximate its time derivative. First,
using \textit{a priori} bounds in reflexive functional spaces, we obtain a weak limit for the approximate sequences. Then we pass to convergence analysis in
order to prove that the limit function  satisfies (\ref{KB_17}).

This kind of approach, known also as the Rothe method, has been used for solving many types of evolution partial differential equations or variational
inequalities. We refer to \cite{Roubicek2005} as for a basic handbook concerning this subject. The Rothe method for evolution inclusion has been applied first
in \cite{Kalita2013} and then developed in \cite{Bartosz_theta, bxkyz, Bartosz_Sofonea, Kalita2, Kalita3, Peng1, Peng2, Peng3}.

There are two main difficulties arising in our problem, both concern the analysis of convergence. The first issue comes from the fact that the operator $A$ is
assumed to be pseudomonotone with respect to the second variable, which is a relatively weak assumption in comparison to \cite{Emrichh_2010}. Moreover, we have
to provide an analogous property of its Nemytskii operator $\mathcal{A}$. To this end, we use Lemma \ref{kalita}, which requires to know that the considered
sequence of piecewise constant functions is bounded in space $M^{p,q}(0,T; W,W^*)$, thus, in particular that they have a bounded  total variation.

The second main difficulty appears when passing to the limit with multivalued term. In this part, we use the Aubin-Celina convergence theorem. However, to do
this, we need to have a strong convergence of the sequence $\gamma v_\tau$ in an appropriate space, where the functions are piecewise constant, and in a
consequence, their time derivatives are not regular enough to apply classical Lions-Aubin compactness results in our case. Thus, we apply more general result
of Lemma \ref{kalita_compactness}, which requires only that functions $v_\tau$ have bounded total variations, instead of bounded time derivative in a space of
type $L^q$ with respect to time. We remark that the compactness of the operator $\gamma$ is a crucial assumption, which allows to use Lemma
\ref{kalita_compactness}. In examples $\gamma$ is either the compact embedding $W_0^{1,p}(\Omega)\subset L^p(\Omega)$ or the trace operator $\gamma:
W^{1,p}(\Omega)\to L^p(\partial \Omega)$.

The rest of the paper is organized as follows. In Section \ref{Preliminaries}, we introduce basic definitions and recall some useful results. In Section
\ref{Statement}, we formulate an abstract problem and establish assumptions on its data. In Section \ref{Discrete_problem}, we state a discrete problem and
obtain \textit{a priori} estimates on its solution. In Section \ref{Convergence_of_numerical_scheme}, we study convergence of solutions of the discrete problem
to a solution of exact one. Finally, in Section \ref{Examples}, we show two examples for which our main result is applicable.

%%%%%%%%%%%%%%%%%%%%%%%%%%%%%%%%%%%

\section{Preliminaries}\label{Preliminaries}
In this section we introduce the basic definitions and recall results used in the sequel. We start with the definition of a pseudomonotone operator in both
single valued and multivalued case.

\begin{Definition}  Let $X$ be a real Banach space. A single valued operator $A\colon X\to X^*$ is called pseudomonotone, if
  for any sequence $\{v_n\}_{n=1}^\infty\subset X$ such that $v_n\to v$ weakly in $X$ and $\limsupn\langle Av_n, v_n-v\rangle\leqslant 0$ we have
  $\langle Av, v-y\rangle \leqslant \liminfn\langle Av_n, v_n-y\rangle$ for every $y\in X$.
\end{Definition}

 \begin{Definition}
 Let $X$ be a real Banach space. The multivalued operator $A\colon X\to 2^{X^*}$ is called pseudomonotone if the following conditions hold:
\begin{itemize}
  \item[1)] $A$ has values which are nonempty, weakly compact and convex,
  \item[2)] $A$ is upper semicontinuous from every finite dimensional subspace of $X$ into $X^{\ast}$ furnished with weak topology,
  \item[3)] if $\{v_n\}_{n=1}^\infty\subset X$ and $\{v_n^*\}_{n=1}^\infty\subset X^*$ are two sequences such that
  $v_n \rightarrow v$ weakly in $X$, $v_n^{\ast} \in A(v_n)$ for all $n\geqslant 1$ and
  $\limsupn \langle v_n^\ast, v_n - v \rangle \leqslant 0$,
  then for every $y \in X$ there exists $u(y) \in A(v)$ such that
  $\langle u(y), v - y \rangle \leqslant \liminfn \langle v_n^{\ast}, v_n -
  y\rangle$.
\end{itemize}
\end{Definition}
Now we recall two important results concerning properties of pseudomonotone operators.
\begin{Proposition}\label{prop:sum_pseudo}
  Assume that  $X$ is a reflexive Banach space and
  $A_1, A_2\colon X\to 2^{X^*}$ are pseudomonotone operators.
  Then  $A_1+A_2\colon X\to 2^{X^*}$
  is a pseudomonotone operator.
\end{Proposition}

\begin{Theorem} \label{pseudo}
  Let $X$ be a reflexive Banach space and
  let $A\colon X\to 2^{X^*}$ be a pseudomonotone and coercive operator. Then $A$ is surjective, i.e. $R(A)=X^*$.
\end{Theorem}

Let $X$ be a Banach space and let $I=(0,T)$ be a time interval. We introduce the space $BV(I;X)$ of functions of bounded total variation on $I$. Let $\pi$
denote any finite partition of $I$ by a family of disjoint subintervals $\{\sigma_i = (a_i,b_i)\}$ such that $\bar{I} = \cup_{i=1}^n \bar{\sigma}_i$. Let
$\mathcal{F}$ be the family of all such partitions. Then for a function $x\colon I\to X$ we define its total variation by
\[
  \| x \|_{BV(I;X)} = \sup_{\pi \in \mathcal{F}} \bigg\{  \sum_{\sigma_i \in \pi} \|x(b_i)-x(a_i)\|_X \bigg\}.
\]
As a generalization of the above definition, for $1 \leqslant q < \infty$, we define a seminorm
\[
  \| x \|^q_{BV^q (I;X)} = \sup_{\pi \in \mathcal{F}} \bigg\{ \sum_{\sigma_i \in \pi} \|x(b_i)-x(a_i)\|_X^q \bigg\}
\]
and the space
\[
  BV^q(I;X)=\{x\colon I\to X;\ \| x \|_{BV^q (I;X)}<\infty\}.
\]
For $1 \leqslant p \leqslant \infty$, $1\leqslant q < \infty$ and Banach spaces $X$ and $Z$ such that $X \subset Z$, we introduce the following space
\[
   M^{p,q}(I;X,Z) = L^p(I;X) \cap BV^q (I;Z).
\]
Then $M^{p,q}(I; X,Z)$ is also a Banach space with the norm given by $\| \cdot \|_{L^p(I;X)} + \| \cdot \|_{BV^q(I;Z)}$.

 Finally, we recall a compactness result, which will be used in the sequel. For its proof, we refer to \cite{Kalita2013}.

\begin{Proposition}\label{kalita_compactness}
Let $1\leqslant p,q<\infty$. Let $X_1\subset X_2\subset X_3$ be real Banach spaces such that $X_1$ is reflexive, the embedding $X_1\subset X_2$ is compact and
the embedding $X_2\subset X_3$ is continuous. Then the embedding $M^{p,q}(0,T;X_1;X_3)\subset L^p(0,T;X_2)$ is compact.
\end{Proposition}

\section{Problem statement}\label{Statement}
  In this section we formulate an abstract problem and give a list of assumptions concerning the data of the problem.
  For a Banach space $X$ by $X^*$ we denote its topological dual,
  by $\dual{\cdot,\cdot}{X^*\times X}$ the duality pairings for the pair $(X,X^*)$ and
  by $i_{XY}\colon X\to Y$ we will denote the embedding operators of $X$ into $Y$ provided that $X\subseteq Y$.

  First we introduce appropriate spaces.
  Let $(W,\|\cdot\|_W)$ be a reflexive Banach space densely and continuously embedded in a reflexive Banach space
  $(V,\|\cdot\|_V)$, and let $(V,\|\cdot\|_V)$
  be densely and continuously embedded in a Hilbert space $(H,(\cdot, \cdot), |\cdot|)$. We also assume that the embedding $W\subseteq H$ is compact. We have
\[
   W\subseteq V \subseteq H\subseteq V^* \subseteq W^*,
\]
where $V^*$ and $W^*$ denote the dual spaces to $V$ and $W$, respectively. Let $(U,\|\cdot\|_U)$ be a Banach space such that there exists a compact mapping
$\gamma\colon W\to U$. For $T>0$, $p\geqslant 2$ we define the spaces $\mathcal{W}=L^p(0,T;W)$, $\mathcal{V}=L^p(0,T;V)$, ${\mathcal{H}}=L^2(0,T;H)$,
$\mathcal{U}=L^p(0,T;U)$. We knot that their dual spaces are $\mathcal{W}^*=L^q(0,T;W^*)$, $\mathcal{V}^*=L^q(0,T;V^*)$, ${\mathcal{U}}^*=L^q(0,T;U^*)$,
respectively (where $\frac{1}{p} + \frac{1}{q} =1$). We identify the space $\mathcal{H}$ with its dual and denote by $(\cdot,\cdot)_\mathcal{H}$ the scalar
product in $\mathcal{H}$.

We are concerned with the following problem.\\

\noindent {\bf Problem {${\mathcal{P}}$}}. Find $u\in \mathcal{W}$ with $u'\in\mathcal{W}$ and $u''\in\mathcal{W}^*$ such that
\begin{eqnarray}
 && \!\!\!\!\!\!\!\!\!\!  u''(t) + A(t,u'(t)) + B(t,u(t)) + \gamma^* M(\gamma u'(t)) \ni f(t)\ \mbox{a.e.}\ {t\in (0,T)},  \label{original1} \\
 && \!\!\!\!\!\!\!\!\!\!  u(0)=u_0, \quad
u'(0)=v_0.\label{original2}
\end{eqnarray}

\noindent A solution of {\bf Problem {${\mathcal{P}}$}} will be understood in the following sense.

\begin{Definition}\label{def_solution}
The function $u\in{\mathcal{W}}$ is said to be a solution of Problem {${\mathcal{P}}$} if $u'\in\mathcal{W}$, $u''\in\mathcal{W}^*$, $u$ satisfies
\eqref{original2}, and there exists a function $\eta\in {\mathcal{U}^*}$ such that
\begin{eqnarray}
 && \!\!\!\!\!\!\!\!\!\!  u''(t) + A(t,u'(t)) + B(t,u(t)) + \gamma^*\eta(t)= f(t)\ \mbox{a.e.}\ {t\in (0,T)},  \label{original3} \\
 && \!\!\!\!\!\!\!\!\!\!  \eta(t)\in M(\gamma u'(t))\ \mbox{a.e.}\ {t\in (0,T)}.\label{original4}
\end{eqnarray}
\end{Definition}

We impose the following assumptions on the data of Problem {$\mathcal{P}$}.\\

\noindent ${\underline{H(A)}:} A\colon [0,T]\times W\to W^*$ is such that
\begin{itemize}
  \item[(i)] for all $v\in W$, the mapping $t\to A(t,v)$ is continuous,
  \item[(ii)] $\|A(t,v)\|_{W^*} \leqslant \beta_A \big(1+\|v\|_W^{p-1}\big)$ for a.e. $t\in (0,T)$, all $v\in W$ with $\beta_A >0$,
  \item[(iii)] $\langle A(t,v),v\rangle_{W^*\times W} \geqslant \mu_A \|v\|_W^p -\beta |u|^2 -\lambda$
             for all $v\in W$ with $\mu_A >0$, ${\beta,\lambda \in \mathbb{R}}$,
  \item[(iv)] $v\to A(t,v)$ is pseudomonotone for all $t\in[0,T]$.
\end{itemize}

We assume that $B\colon [0,T]\times V\to W^*$ has a decomposition $B(t,v)=B_0(v) + C(t,v)$, where\\

\noindent ${\underline{H(B_0)}:} B_0\colon \in\mathcal{L}(V,V^*)$ is symmetric and strongly positive, with constants $\mu_B,\beta_B>0$ such that
\[
  \langle B_0 v,v\rangle \geqslant \mu_B \|v\|_V^2, \qquad \|B_0 v\| \leqslant \beta_B \|v\|_V.
\]

\noindent ${\underline{H(C)}:} C\colon [0,T]\times V\to W^*$ is such that
\begin{itemize}
  \item[(i)] for all $v\in V$, the function $t\to C(t,v)$ is continuous,
  \item[(ii)] $\|C(t,v)\|_{W^*} \leqslant \beta_C(1+\|v\|_V^{\frac{2}{q}})$  for a.e. $t\in (0,T)$, all $v\in V$ with $\beta_C >0$,
  \item[(iii)] $\|C(t,v)-C(t,w)\|_{W^*}\leqslant \alpha(\max(\|v\|_V,\|w\|_V))|v-w|^{\frac{1}{q}}$
              for all $t\in [0,T]$, all $v,w\in V$, where $\alpha\colon \mathbb{R}_+ \to \mathbb{R}_+ $ is a monotonically increasing function.
\end{itemize}

\noindent ${\underline{H(M)}:} M\colon U\to 2^{U^*}$ is such that

\begin{itemize}
 \item[(i)] for all $u\in U$, $M(u)$ is a nonempty, closed and convex set,
 \item[(ii)] $M$ is upper semicontinuous in $(s\textrm{-}U\times w\textrm{-}U^*)$-topology,
 \item[(iii)] $\|\eta\|_{U^*} \leqslant c_M(1+\|w\|_U^{p-1})$ for all $w\in U$, all $\eta \in M(w)$.
\end{itemize}

\noindent ${\underline{H(f)}}$ $f\in {\mathcal{W}}^*$.\\

\noindent ${\underline{H(\gamma)}}$: $\gamma\colon  W\to U$ is linear, continuous and compact and its Nemytskii operator $\bar{\gamma}\colon M^{p,q}(0,T;W,W^*)\to L^q(0,T;U^*)$ is compact.\\

\noindent ${\underline{H_0}:} \mu_A>c_M\|\gamma\|_{\mathcal{L}(W,U)}^p$.\\

Now, we provide a result concerning pseudomonotonicity of the superposition $\gamma^* M(\gamma)$.

\begin{Lemma}\label{lemma_1}
Let the multivalued operator $M\colon U\to 2^{U^*}$ satisfy assumption $H(M)$ and the operator $\gamma\colon W\to U$ be linear, continuous and compact. Then
the operator $W\ni v\to\gamma^*M(\gamma u)\in W^*$ is pseudomonotone.
\end{Lemma}
The proof of Lemma \ref{lemma_1} can be obtained similarly as the proof of Proposition 5.6 in \cite{HMSBOOK}. \\

We complete this section with a lemma, which will play a crucial role in the convergence of numerical scheme presented below.

\begin{Lemma}\label{kalita}
  Let $A\colon [0,T]\times W\to W^*$ be an operator satisfying hypotheses $H(A)$ and
  $\mathcal{A}\colon \mathcal{W}\to \mathcal{W}^*$ be a Nemytskii operator corresponding to $A$
  defined by $(\mathcal{A}v)(t)=A(t,v(t))$ for all $t\in[0,T]$, all $v\in \mathcal{W}$.
  Assume that $\{v_n\}\subset  \mathcal{W}$ is a sequence
  bounded in $M^{p,q}(0,T;W,W^*)$ and such that $v_n\to v$ weakly in $\mathcal{W}$ and
\[
  \limsupn\langle{\mathcal{A}v_n,v_n-v}\rangle_{\mathcal{W}^*\times \mathcal{W}}\leqslant 0.
\]
  Then $\mathcal{A}v_n\to \mathcal{A}v$ weakly in $\mathcal{W}^*$.
\end{Lemma}

The proof of Lemma \ref{kalita} can be obtained by standard techniques, cf. Lemma 2 in \cite{Kalita3}.

\section{Discrete problem}\label{Discrete_problem}
  In this section we consider a discrete problem corresponding to Problem {$\mathcal{P}$}.

  For $N\in\nat$ we consider an arbitrary fixed time grid
\[
  0=t_0 <t_1 \ldots t_{n-1}  <t_n^N = T, \quad \tau_n = t_n-t_{n-1}\quad\textrm{for}\ n=1,\ldots,N.
\]

We define the following discretization parameters
\[
  \tau_{n+\frac{1}{2}} = \frac{\tau_n + \tau_{n+1}}{2}, \quad t_{n+\frac{1}{2}}=t_n + \frac{1}{2}\tau_{n+1}
  \quad\textrm{for}\ n=1,\ldots,N-1,
\]
\[
  \quad r_{n+1}=\frac{\tau_{n+1}}{\tau_n}
  \quad\textrm{for}\ n=1,\ldots,N-1,
\]
\[
  \gamma_n:=\max\left(0,\frac{1}{r_n}-\frac{1}{r_{n-1}}\right)\quad\textrm{for}\ n=2,\ldots,N,
\]
\[
  \tau_{max}=\max_{n=1,\ldots,N} \tau_n, \quad r_{max}:=\max_{n=2,\ldots,N}r_n, \quad r_{min}:=\max_{n=2,\ldots,N}r_n,
\]
\[
  c_\gamma:=\max_{n=3,\ldots,N} \frac{\gamma_n}{\tau_n},\quad
  \sigma(\tau)=\frac{1}{2}\sum_{j=1}^{N-1}\frac{(\tau_{j+1}-\tau_j)^2}{\tau_{j+1}+\tau_j}.
\]
  We also define  $f^n=\frac{1}{\tau_{n+\frac{1}{2}}}\int_{t_{n-\frac{1}{2}}}^{t_{n+\frac{1}{2}}}f(t)\,dt$ for $n=1,...,N-1$.\\
  Finally, in order to approximate the initial conditions, we introduce elements $u^0_\tau$, $v^0_\tau$, whose convergence to $u_0$ and $v_0$ will be specified later.\\

  The discrete problem reads as follows.\\

\noindent {\bf Problem ${\mathcal{P}}_\tau$}. Find sequences $\{u^n\}_{n=0}^N\subset W$ and $\{v^n\}_{n=0}^N\subset W$ such that
\begin{eqnarray}
 && v^n=\frac{1}{\tau_{n+1}} (u^{n+1}-u^n), \quad n=0,1,\ldots,N-1, \label{dis1}\\
 && \frac{1}{\tau_{n+\frac{1}{2}}}
     (v^n-v^{n-1}) + A(t_n,v^n) + B(t_n,u^n) + \gamma^*\eta^n = f^n,\cr
 & & \hspace*{6cm} \quad n=1,2,\ldots,N-1,  \label{discrete} \\
 && \eta^n \in M(\gamma v^n), \label{dis2}\\
 && u^0 = u^0_\tau, \quad v^0 = v^0_\tau. \label{dis3}
\end{eqnarray}

  In what follows, we formulate a theorem concerning existence of solution to Problem ${\mathcal{P}}_\tau$.

\begin{Theorem}
  Under hypotheses $H(A)$, $H(B_0)$, $H(C)$, $H(M)$, $H_0$ and $\tau_{max}<\frac{1}{\beta}$
  there exist sequences $\{u^n\}_{n=0}^N$ and $\{v^n\}_{n=0}^N$ being a solution
  to Problem ${\mathcal{P}}_\tau$.
\end{Theorem}

\begin{proof}
  We define the multivalued operator $T\colon W\to 2^{W^*}$ by
\[
  Tv = \frac{1}{\tau_{n+\frac{1}{2}}} v + A(t_n,v) + \gamma^* M(\gamma v),\ \textrm{for}\ v\in W.
\]
  First we show that $T$ is coercive. Let $v\in W$ and $z\in Tv$. Thus we have $z=\frac{1}{\tau_{n+\frac{1}{2}}} v + A(t_n,v) + \gamma^*\eta$
  with $\eta \in M(\gamma v)$. Using hypotheses $H(A)$, $H(B_0)$, $H(C)$ and $H(M)$, we estimate
\begin{eqnarray*}
 & &
   \langle z,v\rangle_{W^*\times W} = \frac{1}{\tau_{n+\frac{1}{2}}} (v,v)_H + \langle A(t_n,v),v\rangle_{W^*\times W} + \langle \eta,\gamma v\rangle_{U^*\times U}\cr
 & \geqslant &
   \frac{1}{\tau_{n+\frac{1}{2}}}|v|_H^2 + \mu_A\|v\|_W^p - \beta|v|_H^2 -\lambda -\|\eta\|_{U^*}\|\gamma v\|_U \cr
 & \geqslant &
     \bigg(\frac{1}{\tau_{n+\frac{1}{2}}}-\beta\bigg)|v|_H^2 + \mu_A\|v\|_W^p  -\lambda - c_M\left(1+\|\gamma v\|_U^{p-1}\right)\|\gamma v\|_U \cr
  & \geqslant & \bigg(\frac{1}{\tau_{n+\frac{1}{2}}}-\beta\bigg)|v|_H^2 + \mu_A\|v\|_W^p\cr
  & &  -\lambda -c_M\|\gamma\|^p_{\mathcal{L}(W,U)}\|v\|^p -c_M\|\gamma\|_{\mathcal{L}(W,U)} \|v\|_W \cr
  & \geqslant & \bigg(\frac{1}{\tau_{n+\frac{1}{2}}}-\beta\bigg)|v|_H^2 + \left(\mu_A - c_M \|\gamma\|_{\mathcal{L}(W,U)}^p \right)\|v\|_W^p\cr
  & & -\lambda -c_M\|\gamma\|_{\mathcal{L}(W,U)} \|v\|_W.
\end{eqnarray*}
  Using $H_0$ and inequality $\tau_{max}<\frac{1}{\beta}$, we see that $T$ is coercive. From $H(A)(iv)$ and Lemma \ref{lemma_1}, we conclude that
  operator $T$ is pseudomonotone as a sum of three pseudomonotone operators. This allows to use
  Theorem~\ref{pseudo} to conclude that $T$ is surjective, and as a result, we can establish the existence of $v^n$ for a given $v^0,\ldots,v^{n-1}$ in Problem ${\mathcal{P}}_\tau$.
  Moreover, using
\begin{eqnarray*}
  u^n
  & = & u^0+\sum_{j=0}^{n-1} (u^{j+1}-u^j)\cr
  & = & u^0 +\sum_{j=0}^{n-1} \tau_{j+1} v^j:=L(v^n), \quad n=0,1,\ldots, N,
\end{eqnarray*}
  we can recover the sequence $u^1, u^2,...,u^n$. This completes the proof.
\end{proof}

  The next lemma concerns \textit{a priori} estimate for solution of Problem ${\mathcal{P}}_\tau$. In what follows, we denote by $c$ a constant independent on $\tau$, which can vary from line to line. The dependence of $c$ on the other data or parameter will be specified if needed.    

\begin{Lemma}[\textit{A priori} estimate] \label{apriori}
Let hypotheses $H(A)$, $H(B_0)$, $H(C)$, $H(M)$, $H_0$ hold and the time grid satisfy the following constraint
\begin{equation}\label{eq_KB_1}
  \tau_{max} < \min\Bigg\{\frac{2\left(\mu_A-c_M\|\gamma\|_{\mathcal{L}(W,U)}^p\right)}{\beta_B \|i_{WV}\|_{\mathcal{L}(W,V)}}, \frac{1}{2\beta}\Bigg\}.
\end{equation}
Then, for $n=1,2,\ldots,N-1$, we have
\begin{eqnarray}\label{ap1}
 & & \|u^{n+1}\|_V^2 + |v^n|^2 + \sum_{j=1}^n |v^j-v^{j-1}|^2 + \sum_{j=1}^n \tau_{j+\frac{1}{2}} \|v^j\|_W^p +  \sum_{j=1}^n \tau_{j+\frac{1}{2}} \|\eta^j\|_{U^*}^q \cr
 & \leqslant & c \bigg(1+\|u^0\|_V + |v^0|^2 + \tau_1^2\|v^0\|_V + \sum_{j=1}^n \tau_{j+\frac{1}{2}} \|f^j\|_{W^*}^q \bigg),
\end{eqnarray}
where $c=c(r_{min},r_{max},c_\gamma,T)>0$. Moreover
\begin{equation}
\sum_{j=1}^n \tau_{j+\frac{1}{2}} \bigg\|\frac{1}{\tau_{j+\frac{1}{2}}} (v^j - v^{j-1})\bigg\|_{W^*}^q \leqslant c. \label{ap2}
\end{equation}
\end{Lemma}

\begin{proof}
  We test \eqref{discrete} by $v^n$ and calculate
\begin{eqnarray}\label{1}
  \frac{1}{\tau_{j+\frac{1}{2}}}(v^n -v^{n-1},v^n)
  & = & \frac{1}{2 \tau_{j+\frac{1}{2}}}\left(|v^n|^2-|v^{n-1}|^2 + |v^n-v^{n-1}|^2\right)\\
  \langle A(t_n,v^n),v^n\rangle_{W^*\times W}
  & \geqslant & \mu_A\|v^n\|_W^p -\beta |v^n|^2 -\lambda\\
  \langle B(t_n,u^n),v^n \rangle_{W^*\times W}
  & = & \langle B_0(u^n),v^n\rangle + \langle C(t,u^n),v^n\rangle.
\end{eqnarray}
  We introduce the inner product $\langle\cdot,\cdot\rangle_B \colon V\times V\to \mathbb{R}$ by
\[
  \langle u,v\rangle_B:= \langle B_0u,v\rangle_{V^*\times V}\quad\textrm{for}\ u,v\in V
\]
  and the corresponding norm
\[
  \|u\|_B=\sqrt{\langle u,u\rangle_B}\quad\textrm{for}\ u\in V.
\]
  Note that the norms $\|u\|_B$ and $\|u\|_V$ are equivalent since $\mu_B\|u\|_V^2 \leqslant \|u\|_B^2 \leqslant \beta_B\|v\|_V^2$ for all $u\in V$.
  We have
\begin{eqnarray}
 & & \langle B_0 u^n,v^n \rangle_{V^*\times V} = \langle B_0 Lv^n,v^n\rangle_{V^*\times V}\cr
 & = & \big\langle B_0 Lv^n,\frac{1}{\tau_{n+1}}(Lv^{n+1}-Lv^n)\big\rangle_{V^*\times V} \cr
 & = & \frac{1}{2\tau_{n+1}}\left(\|Lv^{n+1}\|_B^2-\|Lv^n\|_B^2-\tau_{n+1}^2 \|v^n\|_B^2\right) \cr
 & = & \frac{1}{2\tau_{n+1}}(\|u^{n+1}\|_B^2-\|u^n\|_B^2-\tau_{n+1}^2 \|v^n\|_B^2 ).
\end{eqnarray}
 From hypotheses $H(C)$, using Young inequality, for any fixed $\varepsilon>0$, we find that
\begin{eqnarray*}
  & &
  |\langle C(t_n,u^n),v^n\rangle_{W^*\times W}| \leqslant \|C(t_n,u^n)\|_{W^*} \|v^n\|_W\\
  &  \leqslant & \varepsilon\|v^n\|_W^p+ c(\varepsilon)\|C(t_n,u^n)\|_{W^*}^q \\
  & \leqslant & \varepsilon\|v^n\|_W^p + c(\varepsilon) (1+\|u^n\|_V^2).
\end{eqnarray*}
  We come to the multivalued term
\begin{eqnarray*}
  & & |\langle \gamma^*\eta^n,v^n\rangle_{W^*\times W}|=|\langle \eta^n,\gamma v^n\rangle_{U^*\times U}| \leqslant \|\eta^n\|_{U^*} \|\gamma v^n\|_U\\
  & \leqslant & c_M(1+\|\gamma v^n\|_U^{p-1})\|\gamma v^n\|\\
  & \leqslant & (c_M\|\gamma\|^p_{\mathcal{L}(W,U)}+\varepsilon) \|v^n\|_W^p + c(\varepsilon) c_M^q \|\gamma\|_{\mathcal{L}(W,U)}^q
\end{eqnarray*}
  and
\begin{equation}
  \label{2} \langle f^n,v^n \rangle\leqslant \varepsilon \|v^n\|_W^p + c(\varepsilon) \|f^n\|_{W^*}^q.
\end{equation}

We test \eqref{discrete} with $v^n$, apply \eqref{1}-\eqref{2}, replace $n$ with $j$ and multiply by $2\tau_{j+\frac{1}{2}}$ to obtain
\begin{eqnarray}\label{poj}
 & & |v^j|^2 - |v^{j-1}|^2 + |v^j-v^{j-1}|^2 + 2\tau_{j+\frac{1}{2}}(\mu_A-c_M\|\gamma\|_{\mathcal{L}(W,U))}^p-3\varepsilon)\|v\|_W^p \cr
 & & -2\tau_{j+\frac{1}{2}}\beta|v^j|^2 + \frac{\tau_{j+\frac{1}{2}}}{\tau_{j+1}} (\|u^{j+1}\|_B^2-\|u^j\|_B^2) -\tau_{j+1}\tau_{j+\frac{1}{2}} \|v^j\|_B^2\cr
 & & -2\tau_{j+\frac{1}{2}} c(\varepsilon)\|u^j\|_V^2 \cr
 & \leqslant & 2\lambda   \tau_{j+\frac{1}{2}}+ 2 \tau_{j+\frac{1}{2}} c(\varepsilon) +
       2\tau_{j+\frac{1}{2}} c(\varepsilon)c_M^q \|\gamma\|_{\mathcal{L}(W,U)}^q\cr
 & & + 2 \tau_{j+\frac{1}{2}} c(\varepsilon)\|f^j\|_{W^*}^q.
\end{eqnarray}
  We sum up \eqref{poj} for $j=1,\ldots,n$, to obtain
\begin{eqnarray}\label{0}
 & & |v^n|^2 + \sum_{j=1}^n |v^j-v^{j-1}|^2 + 2\sum_{j=1}^n \tau_{j+\frac{1}{2}}\left(\mu_A-c_M\|\gamma\|_{\mathcal{L}(W,U)}^p -3\varepsilon\right)\|v^j\|_W^p \cr
 & & +\frac{1}{2}\left(1+\frac{1}{r_{n+1}}\right)\|u^{n+1}\|_B^2 + \frac{1}{2}\sum_{j=2}^n\left(\frac{1}{r_j}-\frac{1}{r_{j+1}}\right) \|u^j\|_B^2 \cr
 & \leqslant &|v^0|^2 + 2\beta\sum_{j=1}^n \tau_{j+\frac{1}{2}} |v^j|^2 + \frac{1}{2}\left(1+\frac{1}{r_2}\right)\|u^1\|_B^2 + \sum_{j=1}^n \tau_{j+\frac{1}{2}}
  \tau_{j+1} \|v^j\|_B^2 \cr
 & & +2c(\varepsilon)\sum_{j=1}^n\tau_{j+\frac{1}{2}}\|u^j\|_V^2 + 2c(\varepsilon) \sum_{j=1}^n \tau_{j+\frac{1}{2}}
\|f^j\|^q_{W^*} + cT.
\end{eqnarray}
  Note that
\begin{eqnarray}\label{3}
  & & \sum_{j=1}^n \tau_{j+\frac{1}{2}} \tau_{j+1}\|v\|_B^2 \leqslant \sum_{j=1}^n \tau_{j+\frac{1}{2}} \tau_{j+1} \beta_B
       \|i_{WV}\|_{\mathcal{L}(W,V)}^2(1+\|v^j\|_W^p) \cr
  & \leqslant & \beta_B\|i_{WV}\|_{\mathcal{L}(W,V)}^2 \sum_{j=1}^n \tau_{j+\frac{1}{2}} \tau_{j+1}
       + \beta_B\|i_{WV}\|_{\mathcal{L}(W,V)}^2 \sum_{j=1}^n \tau_{j+\frac{1}{2}} \tau_{j+1} \|v^j\|_W^p \cr
  & \leqslant & cT  + \beta_B\|i_{WV}\|_{\mathcal{L}(W,V)}^2 \tau_{max} \sum_{j=1}^n \tau_{j+\frac{1}{2}} \|v^j\|_W^p
\end{eqnarray}
  and
\begin{eqnarray}\label{4}
  & & \frac{1}{2} \sum_{j=2}^n \left(\frac{1}{r_j}-\frac{1}{r_{j+1}}\right)\|u^j\|_B^2 = -\frac{1}{2}\sum_{j=2}^n
    \left(\frac{1}{r_{j+1}}-\frac{1}{r_j}\right) \|u^j\|_B^2 \cr
  & = & -\frac{1}{2}\sum_{j=2}^n \tau_{j+1} \frac{\gamma_{j+1}}{\tau_{j+1}}\|u^j\|_B^2 \geqslant
-\frac{1}{2}\sum_{j=2}^n \tau_{j+1} c_\gamma \beta_B \|u^j\|_V^2
\end{eqnarray}
  From $\frac{\tau_{j+1}}{\tau_j}=r_{j+1}$ is follows that $\tau_j\leqslant \frac{\tau_{j+1}}{r_{min}}$.
  Thus
\begin{eqnarray}\label{5}
 & & 2c(\varepsilon) \sum_{j=1}^n \tau_{j+\frac{1}{2}} \|u^j\|_V^2 = c(\varepsilon)\sum_{j=1}^n(\tau_j + \tau_{j+1})\|u^j\|_V^2\cr
 & \leqslant &
 c(\varepsilon)\sum_{j=1}^n \left(\frac{\tau_{j+1}}{r_{min}}+\tau_{j+1}\right)\|u^j\|_V^2\cr
 & = & c(\varepsilon) \left(\frac{1}{r_{min}}+1\right)\sum_{j=1}^n \tau_{j+1}\|u^j\|_V^2,
\end{eqnarray}
\begin{equation}\label{6}
  2\beta \sum_{j=1}^n \tau_{j+\frac{1}{2}}|v^j|^2 \leqslant 2\beta\tau_{max}|v^n|^2+\beta\left(\frac{1}{r_{min}+1}\right)\sum_{j=1}^{n-1} \tau_{j+1} |v^j|^2
\end{equation}
  and
\begin{equation}\label{7}
  \mu_B\frac{1}{2}\left(1+\frac{1}{r_{n+1}}\right)\|u^{n+1}\|_V^2 \leqslant \frac{1}{2}\left(1+\frac{1}{r_{n+1}}\right)\|u^{n+1}\|_B^2.
\end{equation}
  Using \eqref{3}-\eqref{7} in \eqref{0}, we get
\begin{eqnarray*}
 & & \left(1-2\beta\tau_{max}\right)|v^n|^2 + \frac{1}{2} \mu_B \left(1+\frac{1}{r_{n+1}}\right)\|u^{n+1}\|_B^2 + \sum_{j=1}^n|v^j-v^{j-1}|^2 \cr
 & & +\sum_{j=1}^n \tau_{j+\frac{1}{2}}\left[2(\mu_A-c_M\|\gamma\|_{\mathcal{L}(W,U)}^p-3\varepsilon)-\beta_B\|i_{WV}\|_{\mathcal{L}(W,V)}^2
    \tau_{max}\right]\|u^j\|^p_W\cr
 & \leqslant & |v^0|^2 + \frac{1}{2}\left(1+\frac{1}{r_2}\right)\beta_B \|\tau_1 v^0 + u^0\|_V^2\cr
 & & + \beta \left(\frac{1}{r_{min}}+1\right)\sum_{j=1}^{n-1} \tau_{j+1} |v^j|^2  + cT\cr
 & & +\left[c(\varepsilon)\left(\frac{1}{r_{min}}+1\right)+
   \frac{1}{2}c_\gamma \beta_B\right] \sum_{j=1}^n \tau_{j+1}\|u^j\|_V^2\cr
  & & + 2c(\varepsilon) \sum_{j=1}^n \tau_{j+\frac{1}{2}} \|f^j\|_{W^*}^q.
\end{eqnarray*}
  Using $H_0$ and (\ref{eq_KB_1}), we see that for $\varepsilon >0$ small enough,
  we can use the Gronwall lemma for the last inequality. This, together with
  hypothesis $H(M)(iii)$, gives (\ref{eq_KB_1}).

  As for (\ref{ap1}), we use \eqref{discrete} and get
\begin{eqnarray}\label{8}
  \bigg\|\frac{1}{\tau_{j+\frac{1}{2}}}(v^j-v^{j-1})\bigg\|_{W^*}^q
  & \leqslant & c(\|A(t_j,v^j)\|_{W^*}^q + \|B(t_j,u^j)\|_{W^*}^q\cr
  & & + \|\gamma^*\eta^j\|_{W^*}^q + \|f^j\|_{W^*}^q),
\end{eqnarray}
with a positive constant $c$. From growth conditions on $A$, $B_0$ and $C$, we estimate
\begin{eqnarray}
 & & \|A(t_j,v^j)\|_{W^*}^q \leqslant c(1+\|v^j\|_W^p), \label {9}\\
 & & \|B(t_j,v^j)\|_{W^*}^q \leqslant c(1+\|u^j\|_V^2), \label{10}\\
 & & \|\gamma^*\eta^j\|_{W^*}^q \leqslant \|\gamma\|_{\mathcal{L}(W,U)}^q \|\eta^j\|_{U^*}^q. \label{11}
\end{eqnarray}
Using \eqref{9}-\eqref{11} in \eqref{8}, multiplying \eqref{8} by $\tau_{j+{\frac{1}{2}}}$ and summing up with $j=1,\ldots,n$, we have
\begin{eqnarray}
 & & \sum_{j=1}^n \tau_{j+\frac{1}{2}} \left\|\frac{1}{\tau_{j+\frac{1}{2}}} (v^j - v^{j-1})\right\|_{W^*}^q \leqslant c\biggl(1+\sum_{j=1}^n
  \tau_{j+\frac{1}{2}}\|v^j\|_W^p \cr
 & & +\sum_{j=1}^n \tau_{j+\frac{1}{2}} \|u^j\|_V^2 + \sum_{j=1}^n \tau_{j+\frac{1}{2}}\|\eta\|_{U^*}^q + \sum_{j=1}^n
\tau_{j+\frac{1}{2}}\|f^j\|_{W^*}^q\biggr). \label{12}
\end{eqnarray}
  Finally, using \eqref{ap1}, we get \eqref{ap2} from \eqref{12}. This completes the proof of the lemma.
\end{proof}

  Now, we use the solution $\{u^n\}_{n=0}^N, \{v^n\}_{n=0}^N$ of \eqref{dis1}-\eqref{dis3} to define piecewise constant and piecewise
  linear functions whose convergence will be studied in next section.
\begin{eqnarray*}
& &
u_{\tau}(t):=
\begin{cases}
0   & \mbox{for} \quad t\in[0,t_{\frac{1}{2}}] \\
u^n & \mbox{for} \quad t\in(t_{n-\frac{1}{2}},t_{n+\frac{1}{2}}] \quad n=1,2,\ldots,N-1\\
0   & \mbox{for} \quad t\in(t_{N-\frac{1}{2}},t_N],
\end{cases}
\cr
& & v_{\tau}(t):=
\begin{cases}
v^0 & \mbox{for} \quad t\in[0,t_{\frac{1}{2}}] \\
v^n & \mbox{for} \quad t\in(t_{n-\frac{1}{2}},t_{n+\frac{1}{2}}] \quad n=1,2,\ldots,N-1\\
v^N & \mbox {for} \quad t\in(t_{N-\frac{1}{2}},t_N],
\end{cases}
\cr
& & \hat{v}_{\tau}(t):=
\begin{cases}
v^0   & \mbox{for} \quad t\in[0,t_{\frac{1}{2}}] \\
v^n +\frac{t-t_{n+\frac{1}{2}}}{\tau_{n+\frac{1}{2}}}(v^n-v^{n-1}) & \mbox{for} \quad t\in(t_{n-\frac{1}{2}},t_{n+\frac{1}{2}}]\\
    & \hspace*{0.5cm} n=1,2,\ldots,N-1\\
v^{N-1} & \mbox{for} \quad t\in(t_{N-\frac{1}{2}},t_N],
\end{cases}
\cr
& & \eta_\tau(t):=
\begin{cases}
\eta^0 & \mbox{for} \quad t\in[0,t_{\frac{1}{2}}] \\
\eta^n & \mbox{for} \quad t\in(t_{n-\frac{1}{2}},t_{n+\frac{1}{2}}] \quad n=1,2,\ldots,N-1\\
\eta^N & \mbox {for} \quad t\in(t_{N-\frac{1}{2}},t_N],
\end{cases}
\cr
& & f_\tau(t):=
\begin{cases}
0   & \mbox{for} \quad t\in[0,t_{\frac{1}{2}}] \\
f^n & \mbox{for} \quad t\in(t_{n-\frac{1}{2}},t_{n+\frac{1}{2}}] \quad n=1,2,\ldots,N-1\\
0   & \mbox {for} \quad t\in(t_{N-\frac{1}{2}},t_N].
\end{cases}
\end{eqnarray*}
  Note that the above functions depend on the parameter $N$.
  However,
  for the sake of simplicity, we omit the symbol $N$ in their notation.\\
  It is well known (see Remark 8.15 in \cite{Roubicek2005}) that
\begin{align}
  f_\tau \to f \quad \textrm{in}~\mathcal{W}^*\,\,\,\text{as}\,\,\,N\to\infty. \label{21}
\end{align}

\section{Convergence of the scheme}\label{Convergence_of_numerical_scheme}
  In this section we study the behaviour of sequences $u_\tau$, $v_\tau$, $\hat{u}_\tau$, $\eta_\tau$ and $f_\tau$  with respect to the increasing number of time
  grids $N$. In what follows, all convergences, unless it is specified differently, will be understood with respect to $N\to\infty$. In
  particular, we impose the following, additional assumptions.\\

\noindent ${\underline{H(\tau)}:}$
\begin{itemize}
\item[(1)] $\tau_{max}\to 0$, \item[(2)] $\tau_{max}\leqslant D\tau_{min}$ with a constant $D>0$ independent on $N$, \item[(3)]
$\sigma_\tau \to 0.$
\end{itemize}

\noindent ${\underline{H_1}:}$
\begin{itemize}
  \item[(1)] $u_\tau^0\to u_0$ in $V$,
  \item[(2)] $v_\tau^0\to v_0$ in $H$,
  \item[(3)] $\sup\limits_{N\in\mathbb{N}} \tau_{max} \|v_{\tau}^0\|_V^2 <\infty$.
\end{itemize}

We introduce the integral operator $K\colon \mathcal{V}\to \mathcal{V}$ defined by
\[
  (Kw)(t)\colon =\int_0^t w(s)\,ds \ \textrm{for all}\ w\in\mathcal{V}\quad\textrm{for}\ t\in [0,T].
\]

\begin{Lemma}[Convergences]\label{convergences}
  Under hypotheses $H(A)$, $H(B)$, $H(M)$, $H(\gamma)$, $H(\tau)$, $H_0$, and $H_1$, there exists $u\in C(0,T;V)$ and $v\in \mathcal{W}$ with
  $v'\in\mathcal{W}^*$ such that $u=u_0+Kv$ and for a subsequence, we have
\begin{enumerate}
  \item[(a)] $u_\tau\to u$ weakly$^*$ in  $L^\infty(0,T;V)$,
  \item[(b)] $v_\tau \to v$ weakly in $\mathcal{W}$ and weakly$^*$ in $L^\infty(0,T;H)$,
  \item[(c)] $\hat{v}_\tau \to v$ weakly in $\mathcal{W}$ and weakly$^*$ in $L^\infty(0,T;H)$,
  \item[(d)] $\hat{v}_\tau'\to v'$ weakly in $\mathcal{W}^*$,
  \item[(e)] $Kv_\tau \to Kv$ weakly$^*$ in $L^\infty(0,T;W)$,
  \item[(f)] $u_0 + Kv_\tau -u_\tau \to 0$ in $L^r(0,T;V)$ for $r\in[1,\infty)$,
  \item[(g)] $\eta_\tau \to \eta$ weakly in $\mathcal{U}^*$,
  \item[(h)] $\hat{v}_\tau \to v$ in $L^r(0,T;H)$ for $r\in[1,\infty)$,
  \item[(i)] $v_\tau \to v$ in $L^r(0,T;H)$ for $r\in[1,\infty)$,
  \item[(j)] $u_\tau \to u$ in $L^r(0,T;H)$ for $r\in[1,\infty)$,
  \item[(k)] $u_0 + Kv_\tau \to u$ w $C(0,T;H)$,
  \item[(l)] $v_\tau$ is bounded in $M^{p,q}(0,T;W;W^*)$.
\end{enumerate}
\end{Lemma}

\begin{proof}
  By estimates \eqref{ap1} and \eqref{ap2}, we easily get
\begin{align}
&\|u_\tau\|_{L^\infty(0,T;V)}\leq c,\label{ap3}\\
&\|v_\tau\|_{\mathcal{W}}+\|v_\tau\|_{L^\infty(0,T;H)}\leq c,\label{ap4}\\
&\|\hat{v}_\tau\|_{\mathcal{W}}+\|\hat{v}_\tau\|_{L^\infty(0,T;H)}\leq c,\label{ap5}\\
&\|\hat{v}'_\tau\|_{\mathcal{W}^*}\leq c,\label{ap6}\\
&\|\eta_\tau\|_{\mathcal{U}^*}\leq c,\label{ap7}
\end{align}
  The convergences $(a)$-$(d)$ and $(g)$ follow from \eqref{ap3}-\eqref{ap6} and \eqref{ap7}, respectively.
  However, we need to show that limits obtained in $(b)$ and $(c)$
  coincide. Note that
\begin{eqnarray}
 & & \|\hat{v}_\tau -v_\tau\|_\mathcal{H}^2 = \sum_{j=1}^{N-1} \int_{t_{j-\frac{1}{2}}}^{t_{j+\frac{1}{2}}}
   \bigg|\frac{t-t_{j+\frac{1}{2}}}{\tau_{j+\frac{1}{2}}}(v^j-v^{j-1})\bigg|^2 \,dt \cr
 & = & \sum_{j=1}^{N-1}
   \bigg|\frac{v^j-v^{j-1}}{\tau_{j+\frac{1}{2}}}\bigg|^2 \int_{t_{j-\frac{1}{2}}}^{t_{j+\frac{1}{2}}} (t-t_{j+\frac{1}{2}})^2 \,dt = \frac{1}{3}
   \sum_{j=1}^{N-1} \bigg|\frac{v^j-v^{j-1}}{\tau_{j+\frac{1}{2}}}\bigg|^2 \tau_{j+\frac{1}{2}} \cr
 & \leqslant & \frac{1}{3}\tau_{max} \sum_{j=1}^{N-1}
\left|v^j-v^{j-1}\right|^2 \to 0.\label{15}
\end{eqnarray}
  So $\hat{v}_\tau -v_\tau \to 0$ in $\mathcal{H}$. Now since $v_\tau-\hat{v}_{\tau} \to v-\hat{v}$
  weakly in $\mathcal{W}$, it follows that
  $v_\tau-\hat{v}_{\tau} \to v-\hat{v}$ weakly in $\mathcal{H}$. From the uniqueness of the limit we obtain $v=\hat{v}$.

  Now we prove $(e)$. Let $g\in L^1(0,T;W^*)$. Then
\begin{eqnarray*}
  & & \langle g,K{v}_\tau -Kv \rangle_{L^1(0,T;W^*)\times L^{\infty}(0,T;W)}\cr
  & = &\int_0^T \bigg\langle g(t),\int_0^t(v_\tau(s)-v(s))\,ds\bigg\rangle_{W^*\times W}\,dt\cr
  & = & \int_0^T \int_0^t\left\langle g(t),v_\tau(s)-v(s)\right\rangle_{W^*\times W} \,ds\,dt \cr
  & = & \int_0^T\left(\int_s^T\langle g(s),v_\tau(s)-v(s)\rangle_{W^*\times W}
       \,dt\right)\, ds \cr
  & = & \int_0^T \left\langle\int_s^T g(t)\,dt,v_\tau(s)-v(s)\right\rangle_{W^*\times W} \,ds\cr
  & = & \int_0^T \langle G(s),v_\tau(s)-v(s)\rangle_{W^*\times W} \,ds \cr
  & = & \langle G,v_\tau -v\rangle_{\mathcal{W}^*\times\mathcal{W}} \to 0,
\end{eqnarray*}
  where $G(s)=\int_s^T g(t)dt$ for $s\in [0,T]$.
  Since $G\in L^\infty(0,T;W^*)$, we have $G\in \mathcal{W}^*$.
  This proves $(e)$.

  To prove $(f)$ we first estimate integrals
\begin{eqnarray}
  \nonumber
 & & \int_0^{t_{\frac{1}{2}}} \|u_0+v^0t\|_V^2\,dt \leqslant \int_0^{t_{\frac{1}{2}}} (\|u_0\|_V + t\|v^0\|_V)^2 \,dt\\
 \nonumber
 & \leqslant & \int_0^{t_{\frac{1}{2}}} 2(\|u_0\|_V^2 + t^2\|v^0\|_V^2)\,dt \\
 \label{KB_2}
 & = & 2t_{\frac{1}{2}} \|u_0\|_V^2+\frac{2}{3}\tau_{\frac{1}{2}}^3\|v^0\|_V^2 \leqslant c\tau_{max}(\|u_0\|_V^2 + \|v^0\|_V^2).
\end{eqnarray}
  For $n=1,...,N-1$, we have
\begin{eqnarray*}
 & & I_n = \int_{t_{n-\frac{1}{2}}}^{t_{n+\frac{1}{2}}} \|u_0 + (Kv_\tau)(t)-u_\tau(t)\|_V^2 \,dt \cr
 & = & \int_{t_{n-\frac{1}{2}}}^{t_{n+\frac{1}{2}}} \bigg\|u_0 + v^0\tau_{\frac{1}{2}} + \sum_{j=1}^{n-1}\tau_{j+\frac{1}{2}} v^j + (t-t_{n-\frac{1}{2}})v^n -u^0-\sum_{j=1}^{n-1}\tau_{j+1}v^j\bigg\|_V^2 \,dt\cr
 & = & \int_{t_{n-\frac{1}{2}}}^{t_{n+\frac{1}{2}}} \bigg\|u_0-u^0 + v^0\tau_{\frac{1}{2}} +
     \sum_{j=1}^{n-1}(\tau_{j+\frac{1}{2}}-\tau_{j+1})v^j+(t-t_{n-\frac{1}{2}})v^n\bigg\|_V^2 \, dt \cr
 & \leqslant &c \int_{t_{n-\frac{1}{2}}}^{t_{n+\frac{1}{2}}}\biggl( \|u_0-u^0\|_V^2 +\tau_{\frac{1}{2}}\|v^0\|_V^2 + \biggl(\sum_{j=1}^{n-1}|\tau_{j+\frac{1}{2}}-\tau_{j+1}|\|v^j\|_V\biggr)^2\cr
 & &
 +(t-t_{n-\frac{1}{2}})^2\|v^n\bigr\|_V^2\biggr)\,dt\cr
 & \leqslant &
  c\biggl[\tau_{n+\frac{1}{2}} \|u_0-u^0\|_V^2 +  \tau_{n+\frac{1}{2}} \tau_{\frac{1}{2}}^2 \|v^0\|_V^2\cr
 & &+ \tau_{n+\frac{1}{2}}
    \bigg(\sum_{j=1}^{n-1} |\tau_{j+\frac{1}{2}}-\tau_{j+1}|\|v^j\|_V\bigg)^2+ \frac{1}{3}\tau_{n+\frac{1}{2}}\|v^n\|_V^2 \biggr].
\end{eqnarray*}
  Hence
\begin{eqnarray}\label{KB_4}
  & &
 \sum_{n=1}^{N-1} I_n
 \ \leqslant \ c\bigg[T \|u_0-u^0\|_V^2 + \tau_{max}\|v^0\|_V^2\cr
 & & + T\bigg(\sum_{j=1}^{N-1} |\tau_{j+{\frac{1}{2}}} -\tau_{j+1}|\|v^j\|_V\bigg)^2 + \tau_{max}^2\sum_{n=1}^{N-1}\tau_{n+\frac{1}{2}} \|v^n\|_V^2 \bigg].
\end{eqnarray}
  Finally, we estimate the integral
\begin{eqnarray}\label{KB_5}
  & &
   \int_{t_{N-\frac{1}{2}}}^T \bigg\|u_0+v^0\tau_{\frac{1}{2}} + \sum_{j=1}^{N-1} \tau_{j+\frac{1}{2}} v^j\bigg\|_V^2 \,dt\cr
  & \leqslant & 3\tau_{max} \biggl(\|u_0\|_V^2 + \tau_{max}\|v^0\|_V^2 + \bigg(\sum_{j=1}^{N-1} \tau_{j+\frac{1}{2}}\|v^j\|_V\biggr)^2\bigg)\cr
  & \leqslant & 3\tau_{max}\biggl(\|u_0\|_V^2 + \tau_{max}\|v^0\|_V^2 + N\sum_{j=1}^{N-1} \tau_{j+\frac{1}{2}}^2 \|v^j\|_V^2 \biggr) \cr
  & \leqslant & 3\tau_{max} \bigg(\|u_0\|_V^2 + \tau_{max}\|v^0\|_V^2 + DT \sum_{j=1}^{N-1} \tau_{j+\frac{1}{2}} \|v^j\|_V^2\bigg).
\end{eqnarray}
Now, using Cauchy-Schwartz inequality, we estimate
\begin{eqnarray}\label{KB_6}
 & & \bigg(\sum_{j=1}^{N-1} |\tau_{j+\frac{1}{2}} - \tau_{j+1} |\|v^j\|_V\bigg)^2 = \bigg(\sum_{j=1}^{N-1} \bigg|\frac{\tau_j-\tau_{j+1}}{2}
  \bigg|\|v^j\|_V\bigg)^2\cr
 & = & \bigg(\sum_{j=1}^{N-1} \frac{\tau_j-\tau_{j+1}}{2\sqrt{\tau_{j+\frac{1}{2}}}} \sqrt{\tau_{j+\frac{1}{2}}}
   \|v^j\|_V\bigg)^2\cr
 & \leqslant & \bigg(\sum_{j=1}^{N-1}\frac{(\tau_j-\tau_{j+1})^2}{4\tau_{j+\frac{1}{2}}}\bigg) \bigg(\sum_{j=1}^{N-1} \tau_{j+\frac{1}{2}} \|v^j\|_V^2\bigg)\cr
 & = & \sigma(\tau)\sum_{j=1}^{N-1} \tau_{j+\frac{1}{2}} \|v^j\|_V^2.
\end{eqnarray}
Since $p\geqslant 2$, we have $s^2\leqslant1+s^p$ for all $s\in\real$. Thus, we have
\begin{equation}\label{KB_7}
 \sum_{j=1}^{N-1} \tau_{j+\frac{1}{2}} \|v^j\|_V^2\leqslant\|i_{WV}\|_{\mathcal{L}(W,V)}\bigg(T+\sum_{j=1}^{N-1}\tau_{j+\frac{1}{2}} \|v^j\|_W^p\bigg).
\end{equation}
 Therefore, by (\ref{KB_2})-(\ref{KB_7}) and hypothesis $H(\tau)$ we obtain $u_0+Kv_\tau -u_\tau \to 0$ in $L^2(0,T;V)$. Since $v_\tau$ is bounded in
 $\mathcal{W}$ it is also bounded in $\mathcal{V}$, so $Kv_\tau$ in bounded in $L^\infty (0,T;V)$. Moreover $u_\tau$ is bounded in $L^\infty(0,T;V)$. So
 $u_0+Kv_\tau-u_\tau$ is bounded in $L^\infty (0,T;V)$. For $r>2$, we have
\begin{eqnarray*}
 & & \|u_0+Kv_\tau -u_\tau\|_{L^r(0,T;V)}^r\cr
 & = &\int_0^T \|u_0+Kv_\tau(t) -u_\tau(t)\|_V^{r-2} \|u_0+Kv_\tau(t) -u_\tau(t)\|_V^2\,dt \cr
 & \leqslant & \|u_0+Kv_\tau -u_\tau\|_{L^\infty(0,T;V)} \|u_0+Kv_\tau -u_\tau\|_{L^2(0,T;V)}^2 \to 0.
\end{eqnarray*}
  Therefore $u_0+Kv_\tau-u_\tau \to 0$ in $L^r(0,T;V)$ for all $r\in [1,\infty)$, which completes the proof of $(f)$.

  From $(a)$, $(e)$, $(f)$ and uniqueness of the weak limit in $L^2(0,T;V)$ we obtain
\begin{equation}\label{KB_9}
 u=u_0 + Kv
\end{equation}
  and, in particular, $u\in C(0,T;V)$. From $(c)$, $(d)$, compactness of embedding $W\subset H$ and the Lions-Aubin lemma, we have
\[
  \hat{v}_\tau \to v \quad \textrm{in}\ L^p(0,T;H).
\]
  Again, since $\hat{v}_\tau$ is bounded in $L^\infty (0,T;H)$ it follows that
\[
  \hat{v}_\tau \to v \quad \textrm{in}\ L^r(0,T;H), \qquad \textrm{for all}\ r\in[1;\infty),
\]
  which proves $(h)$.

  From \eqref{15} and $(h)$ we have $v_\tau \to v$ in $L^2(0,T;H)$,
  and also $v_\tau \to v$ in $L^r(0,T;H)$ with $r\in[1,\infty]$, since
  $v_\tau$ is bounded in $L^\infty(0,T;H)$.

  Thus $(i)$ holds. Now, using $(i)$, we calculate,
\begin{eqnarray}
  \nonumber
 & & \|Kv_\tau-Kv\|_{C(0,T;H)}=\max_{t\in [0,T]} \bigg\|\int_0^t v_\tau(s)\,ds - \int_0^t v(s)\,ds\bigg\|_H\\
  \nonumber
 & \leqslant & \max_{t\in [0,T]}\int_0^t \|v_\tau(s)-v(s)\|_H \, ds = \int_0^T \|v_\tau(s)-v(s)\|_H \, ds\\
  \label{KB_8}
 & = & \|v_\tau -v\|_{L^1(0,T;H)}\to 0.
\end{eqnarray}
  From (\ref{KB_9}) we have
\begin{eqnarray}
  \nonumber
 & & \|u-u_\tau\|_{L^r(0,T;H)}=\|u_0+Kv-u_\tau\|_{L^r(0,T;H)}\\
  \nonumber
 & = & \|u_0+Kv_\tau-u_\tau+Kv-Kv_\tau\|_{L^r(0,T;H)}\leqslant \|u_0+Kv_\tau-u_\tau\|_{L^r(0,T;H)}\\
  \nonumber
 & & +\|Kv-Kv_\tau\|_{L^r(0,T;H)}\leqslant\|u_0+Kv_\tau-u_\tau\|_{L^r(0,T;H)}\\
  \label{KB_10}
 & & +T^{\frac{1}{r}}\|Kv-Kv_\tau\|_{C(0,T;H)}.
\end{eqnarray}
  Combining $(f)$, (\ref{KB_8}) and (\ref{KB_10}), we obtain $(j)$.

  Moreover, by (\ref{KB_9}) we have
  $\|u_0+Kv_\tau-u\|_{C(0,T;H)}=\|Kv_\tau-Kv\|_{C(0,T;H)}$. This together with (\ref{KB_8}) gives $(k)$.

  It remains to show $(l)$. Taking into account (\ref{ap4}), it is enough to estimate the seminorm $\|v_\tau\|_{BV^q(0,T;W^*)}$. Since the function $v_\tau$ is
  piecewise constant, the seminorm will be measured by means of jumps between elements of sequence $\{v_\tau^{k}\}_{k=1}^{N}$. Namely, let
  $\{m_i\}_{i=0}^{n}\subset\{1,...,N\}$ be an increasing sequence of numbers such that $m_0=0$, $m_{n}=N$
  and
\begin{align}\label{eq82}
  \|v_\tau\|_{BV^q(0,T;W^*)}^q = \sum^n_{i=1} \| v^{m_i}_\tau - v^{m_{i-1}}_\tau\|^q_{W^*}.
\end{align}
  In what follows, we estimate
\begin{eqnarray}\label{eqBS_22}
  & &  \sum^n_{i=1} \| v^{m_i}_\tau - v^{m_{i-1}}_\tau\|^q_{W^*}\cr
  & \leqslant & \sum^n_{i=1} \bigg( (m_i - m_{i-1} )^{q-1} \sum^{m_i}_{k = m_{i-1}+1 }
  \| v^{k}_\tau - v^{k-1}_\tau\|^q_{W^*}   \bigg) \cr
  & \leqslant & \bigg( \sum^n_{i=1} (m_i - m_{i-1} )^{q-1} \bigg) \bigg( \sum^n_{i=1} \sum^{m_i}_{k = m_{i-1}+1 }  \| v^{k}_\tau -
     v^{k-1}_\tau\|^q_{W^*} \bigg) \cr
 & \leqslant & N^{q-1} \sum^N_{k=1} \| v^{k}_\tau - v^{k-1}_\tau\|^q_{W^*} = N^{q-1} \tau_{j+\frac{1}{2}}^{q} \sum^N_{k=1}
   \bigg\|\frac{v^{k}_\tau - v^{k-1}_\tau}{\tau_{j+\frac{1}{2}}} \bigg\|^q_{W^*} \cr
 & \leqslant & N^{q-1} \tau_{max}^{q-1} \tau_{j+\frac{1}{2}}\sum^N_{k=1} \bigg\|\frac{v^{k}_\tau - v^{k-1}_\tau}{\tau_{j+\frac{1}{2}}} \bigg\|^q_{W^*}\cr
 & \leqslant & N^{q-1} D^{q-1}\tau_{min}^{q-1} \tau_{j+\frac{1}{2}}\sum^N_{k=1} \bigg\| \frac{v^{k}_\tau -v^{k-1}_\tau}{\tau_{j+\frac{1}{2}}} \bigg\|^q_{W^*} \cr
 & \leqslant & \bar{C} T^{q-1} \tau_{j+\frac{1}{2}} \sum^N_{k=1} \bigg\| \frac{v^{k}_\tau - v^{k-1}_\tau}{\tau_{j+\frac{1}{2}}} \bigg\|_{W^*}^q.
\end{eqnarray}
  We now combine \eqref{ap2} with \eqref{eq82} and (\ref{eqBS_22}) to see that $\|v_\tau\|^q_{BV^q(0,T;W^*)}$ is bounded. This completes the proof of the lemma.
\end{proof}

  Now we formulate the existence theorem which is the main result of the paper.

\begin{Theorem}\label{theorem_1}
Let hypotheses $H(A)$, $H(B_0)$, $H(C)$, $H(\gamma)$, $H_0$ hold and $u_0\in V, v_0\in H, f\in L^q(0,T;W^*)$. Then Problem {$\mathcal{P}$} has a solution such
that $u\in C([0,T];V)$.
\end{Theorem}

\begin{proof}
  We define Nemytskii operators
  $\mathcal{A}\colon \mathcal{W}\to \mathcal{W}^*$,
  $\mathcal{B}_0\colon \mathcal{V}\to \mathcal{V}^*$,
  $\mathcal{C}\colon \mathcal{V}\to \mathcal{W}^*$
  and $\bar{\gamma}\colon \mathcal{W}\to \mathcal{U}$, by
 \begin{eqnarray*}
  (\mathcal{A}v)(t)    & = & A(t,v(t))\quad\textrm{for all}\ t\in[0,T], \textrm{all}\ v\in \mathcal{W},\\
  (\mathcal{B}_0 v)(t) & = & B_0 v(t)\quad\textrm{for all}\ t\in[0,T], \ \textrm{all} \ v\in \mathcal{V},\\
  (\mathcal{C} v)(t)   & = & C(t, v(t))\quad \textrm{for all} \ t\in[0,T], \ \textrm{all} \ v\in \mathcal{V},\\
  (\bar{\gamma}v)(t)   & = & \gamma v(t)\quad\textrm{for all} \ tt\in[0,T], \ \textrm{all} \ v\in \mathcal{W}.\\
\end{eqnarray*}
Moreover, we approximate the operators $\mathcal{A}$ and $\mathcal{C}$ by their piecewise constant interpolates given by
\[
(\mathcal{A}_\tau v)(t):=
\begin{cases}
A(t_1,v(t))     & \textrm{for} \quad t\in[0,t_{\frac{1}{2}}] \\
A(t_n,v(t))     & \textrm{for} \quad t\in(t_{n-\frac{1}{2}},t_{n+\frac{1}{2}}] \quad n=1,2,\ldots,N-1\\
A(t_{N-1},v(t)) & \textrm{for} \quad t\in(t_{N-\frac{1}{2}},t_N],
\end{cases}
\]
\[
(\mathcal{C}_\tau v)(t):=
\begin{cases}
  C(t_1,v(t))     & \textrm{for} \quad t\in[0,t_{\frac{1}{2}}] \\
  C(t_n,v(t))     & \textrm{for} \quad t\in(t_{n-\frac{1}{2}},t_{n+\frac{1}{2}}] \quad n=1,2,\ldots,N-1\\
  C(t_{N-1},v(t)) & \textrm{for} \quad t\in(t_{N-\frac{1}{2}},t_N].
\end{cases}
\]

\noindent Let  $u_\tau$, $v_\tau$, $\hat{v}_\tau$, $\eta_\tau$ and $f_\tau$ be the functions defined in Section \ref{Discrete_problem}.
  Now, Problem ${\mathcal{P}}_\tau$ is equivalent to
\begin{eqnarray}
 & & \hat{v}_\tau ' + \mathcal{A}_\tau v_\tau + \mathcal{B}_0 u_\tau +\mathcal{C}_\tau u_\tau + \bar{\gamma}\eta_\tau = f_\tau\quad \textrm{in} \quad
  L^q(0,T;W^*), \label{Nem1}\\
 & & \eta_\tau(t)\in M(\gamma v_\tau(t))\quad \textrm{for a.e.} \,\,\, t\in(0,T).\label{Nem2}
\end{eqnarray}
  We will pass to the weak limit in $\mathcal{W}^*$ with \eqref{Nem1}. From Lemma~\ref{convergences}$ \,\,(d)$, we have
\begin{equation}\label{16}
  \hat{v}_\tau '\to v'\quad  \textrm{weakly in}\  \mathcal{W}^*.
\end{equation}
  Next, from Lemma~\ref{convergences}$(a)$, we obtain
\begin{equation} \label{17}
  u_\tau \to u \quad \textrm{weakly in}\ \mathcal{V}.
\end{equation}
  Thus, by continuity of $B_0$, we also have
\begin{equation} \label{17asas}
  \mathcal{B}_0 u_\tau\to \mathcal{B}_0 u \quad \textrm{weakly in}\  \mathcal{V}^*.
\end{equation}

Next we will show that $\mathcal{C}_\tau u_\tau \to \mathcal{C}u$ in $\mathcal{W}^* $. First we will show that $\mathcal{C}_\tau u\to \mathcal{C}u$ in
$\mathcal{W}^*$. We will use Lebesgue dominated convergence theorem. We show the pointwise convergence, which follows from $H(C)(i)$, namely
\[
  \|\mathcal{C}_\tau u(t)-\mathcal{C}u(t)\|_{W^*} = \|C(t_n,u(t))-C(t,u(t))\|_{W^*} \to 0\quad\textrm{for}\ t\in (0,T).
\]
  Next, we show boundedness, as follows
\begin{eqnarray*}
 & & \|\mathcal{C}_\tau u(t)-\mathcal{C} u(t)\|_{W^*}^q = \|C(t_n),u(t)-C(t,u(t))\|_{W^*}^q \cr
 & \leqslant &
     2^{q-1}\left(\|C(t_n,u(t))\|_{W^*}^q + \|C(t,u(t))\|_{W^*}^q \right) \leqslant 2^{q-1} 2\beta_C\bigg(1+\|u(t)\|_V^{\frac{2}{q}}\bigg)^q \cr
 & \leqslant & 2^{2q-1}\beta_C\left(1+\|u(t)\|_V^2\right)\leqslant c(1+\|u(t)\|_V^2).
\end{eqnarray*}
  The function $t\to c(1+\|u(t)\|_V^2)$ is integrable, because $u\in L^2(0,T;V)$. By the Lebesgue dominated convergence theorem,
  we have $\mathcal{C}_\tau u \to \mathcal{C}u$
  in $\mathcal{W}^*$. From hypotheses $H(C)(iii)$ and Lemma \ref{convergences}, we obtain
\begin{eqnarray*}
 & & \|\mathcal{C}_\tau u_\tau -\mathcal{C}_\tau u\|_{\mathcal{W}^*}^q =\int_0^T\left\|C(t_n,u_\tau(t))-C(t_n,u(t))\right\|_{W^*}^q \,dt \cr
 & \leqslant & \int_0^T\left(\alpha( \max\left\{\|u_\tau(t)\|_V,\|u(t)\|_V\right\}\right))^q |u_\tau(t)-u(t)|\,dt \cr
 & \leqslant & \alpha\left(\max\left\{\|u_\tau\|_{L^\infty(0,T;V)}^q, \|u\|_{L^\infty(0,T;V)}^q \right\}\right) \|u_\tau-u\|_{L^1(0,T;H)} \to 0.
\end{eqnarray*}
  Since $\|\mathcal{C}_\tau u_\tau-\mathcal{C}u\|_{\mathcal{W}^*}\leqslant \|\mathcal{C}_\tau u_\tau-\mathcal{C}_\tau u\|_{\mathcal{W}^*} + \|\mathcal{C}_\tau u
  -\mathcal{C}u\|_{\mathcal{W}^*} \to 0$ we get
\begin{equation}
  \mathcal{C}_\tau u_\tau \to \mathcal{C}u  \quad \textrm{in}\ \mathcal{W}^* \label{19}
\end{equation}
  By Lemma \ref{convergences}$(g)$ and the continuity of $\bar{\gamma}^*$, we infer that
\begin{eqnarray}
  \bar{\gamma}^* \eta_\tau & \to & \bar{\gamma}^*\eta \quad \textrm{weakly in}\ \mathcal{W}^*. \label{20}
\end{eqnarray}
  It remains to show that
\begin{equation}\label{KB_11}
  \mathcal{A}_\tau v_\tau \to\mathcal{A}v\ \textrm{weakly in }\ \mathcal{W}^*.
\end{equation}
  In order to prove (\ref{KB_11}), we proceed in two steps. First, we show that
\begin{equation}\label{KB_12}
  \mathcal{A}_\tau v_\tau -\mathcal{A}v_\tau\to 0\ \textrm{weakly in}\ \mathcal{W}^*.
\end{equation}
  To this end, let $w\in \mathcal{W}$. We define the  function
\[
  h_\tau(t)=\left\langle \left(\mathcal{A}_\tau
  v_\tau\right)(t)-\left(\mathcal{A}v_\tau\right)(t),w(t)\right\rangle_{W^*\times W}\quad\textrm{for}\ t\in (0,T)
\]
  and note that
  $\langle\mathcal{A}_\tau v_\tau -\mathcal{A}v_\tau,w\rangle_{\mathcal{W}^*\times \mathcal{W}}=\int_0^Th_\tau(t)\,dt$.
  Let $S\subset [0,T]$ be a set of measure zero, such that the function $w$ is well defined on
  the set $[0,T]\setminus S$. Let $t\in[0,T]\setminus S$ and $n\in\nat$ be such that $t\in [t_{n-\frac{1}{2}},t_{n+\frac{1}{2}}]$.
  We estimate
\begin{eqnarray*}
 |h_\tau(t)|
 & = & |\langle A(t_n,v_\tau(t))-A(t,v_\tau(t)),w(t)\rangle_{W^*\times W}|\\
 & \leqslant & \| A(t_n,v_\tau(t))-A(t,v_\tau(t))\|_{W^*} \|w(t)\|_W.
\end{eqnarray*}
By hypothesis $H(\tau)$, it is clear that $t_n\to t$. Thus, by hypothesis $H(A)(i)$, we have $\| A(t_n,v_\tau(t))-A(t,v_\tau(t))\|_{W^*} \to 0$, so $h_\tau
(t)\to 0$ for a.e. $t\in [0,T]$. Moreover, we have
\begin{eqnarray*}
 & & |h_\tau(t)|\leqslant \| A(t_n,v_\tau(t))-A(t,v_\tau(t))\|_{W^*} \|w(t)\|_W  \\
 & \leqslant & \big( \|A(t_n,v_\tau(t))\|_{W^*} + \|A(t,v_\tau(t))\|_{W^*}\big) \ \|w(t)\|_W\\
 & \leqslant & \big(2\beta_A + 2\beta_A\|v_\tau(t)\|_{W^*}^{p-1}\big) \|w(t)\|_W \\
 & = & 2\beta_A \|w(t)\|_W +
    2\beta_A \|v_\tau(t)\|_W^{p-1} \|w(t)\|_W.
\end{eqnarray*}
  By the H\"older inequality, the right hand side is integrable on $[0,T]$, so we can use Lebesgue dominated convergence theorem and we
  $\langle\mathcal{A}_\tau v_\tau - \mathcal{A}v_\tau,w\rangle_{\mathcal{W}^*\times\mathcal{W}}\to 0. $ Since the function $w$ is arbitrary, we obtain
 (\ref{KB_12}).

 In the second step, we calculate
\begin{eqnarray}
  \nonumber
 & & \limsup \langle \mathcal{A}v_\tau,v_\tau -v\rangle_{\mathcal{W}^*\times\mathcal{W}}\\
 \nonumber
 & \leqslant & \limsup \langle \mathcal{A}_\tau  v_\tau,v_\tau-v\rangle_{\mathcal{W}^*\times\mathcal{W}}
     + \limsup\langle\mathcal{A}v_\tau-\mathcal{A}_\tau v_\tau,v_\tau -v\rangle_{\mathcal{W}^*\times\mathcal{W}}\\
 \nonumber
 & \leqslant & \limsup\langle \mathcal{A}_\tau v_\tau,v_\tau -v\rangle_{\mathcal{W}^*\times\mathcal{W}}
       + \limsup \langle\mathcal{A} v_\tau-\mathcal{A}_\tau v_\tau,v_\tau\rangle_{\mathcal{W}^*\times\mathcal{W}}\\
  \label{222}
 & & + \limsup\langle \mathcal{A}_\tau v_\tau- \mathcal{A}v_\tau,v\rangle_{\mathcal{W}^*\times\mathcal{W}}.
\end{eqnarray}
  Using \eqref{KB_12}, we have
\begin{equation}
\limsup\langle \mathcal{A}_\tau v_\tau- \mathcal{A}v_\tau,v\rangle_{\mathcal{W}^*\times\mathcal{W}} =0. \label{23}
\end{equation}
Analogously as in the proof of \eqref{23}, we show that
\begin{equation}
\limsup \langle\mathcal{A} v_\tau-\mathcal{A}_\tau v_\tau,v_\tau\rangle_{\mathcal{W}^*\times\mathcal{W}} =0. \label{24}
\end{equation}
From \eqref{Nem1} we get
\begin{eqnarray}
  \nonumber
 & & \hspace{-1.0cm}\limsup\langle \mathcal{A}_\tau v_\tau,v_\tau -v\rangle_{\mathcal{W}^*\times\mathcal{W}} = \langle f_\tau,v_\tau -v \rangle_{\mathcal{W}^*\times\mathcal{W}} +
    (\hat{v}_\tau ',v-v_\tau)_\mathcal{H}\\
 & & \hspace{-1.0cm}+ \langle\mathcal{B}_0 u_\tau,v-v_\tau\rangle_{\mathcal{V}^*\times\mathcal{V}} - \langle\mathcal{C}_\tau
  u_\tau,v_\tau-v\rangle_{\mathcal{W}^*\times\mathcal{W}}-\langle\eta_\tau,\bar{\gamma}v_\tau-\bar{\gamma} v\rangle_{\mathcal{U}^*\times\mathcal{U}}. \label{25}
\end{eqnarray}
  From \eqref{21} and Lemma~\ref{convergences}(d), we have
\begin{equation}
  \langle f_\tau,v_\tau -v\rangle_{\mathcal{W}^*\times\mathcal{W}} \to 0.
\end{equation}
  Moreover, we have
\begin{eqnarray*}
 & & \limsup ( \hat{v}_\tau ',v-v_\tau  )_\mathcal{H} = \limsup\left( ( \hat{v}_\tau ',v )_\mathcal{H}
     - (\hat{v}_\tau ',\hat{v}_\tau  )_\mathcal{H} +  (\hat{v}_\tau',\hat{v}_\tau -v_\tau )_\mathcal{H} \right) \\
 & \leqslant & \lim (\hat{v}_\tau ',v )_\mathcal{H} -\liminf  (\hat{v}_\tau  ',\hat{v}_\tau )_\mathcal{H} + \limsup (
    \hat{v}_\tau  ',\hat{v}_\tau -v_\tau  )_\mathcal{H} =  ( v',v )_\mathcal{H} \\
 & &
  -\liminf\bigg(\frac{1}{2}|\hat{v}_\tau (T)|^2-\frac{1}{2}|\hat{v}_\tau(0)|^2\bigg) +
  \limsup\bigg(-\frac{1}{2}\sum_{j=1}^{N-1}\big(t_{j+\frac{1}{2}}-t_{j-\frac{1}{2}}\big)^2\bigg)\\
 & \leqslant & \frac{1}{2}\big(|v(T)|^2-|v(0)|^2 + \limsup |\hat{v}_\tau(0)|^2 -\liminf |\hat{v}_\tau(T)|^2\big) \\
  & \leqslant & \frac{1}{2}\big( |v(T)|^2-\liminf|\hat{v}_\tau(T)|^2 +\lim |v^0|^2 -|v(0)|^2\big).
\end{eqnarray*}
  From Lemma~\ref{convergences}(c) and (d) and from continuity of the embedding
\[
  \{v\in L^p(0,T;W)\mid v'\in L^q(0,T;W^*)\}\subset C(0,T;H),
\]
  we have $v_\tau \to v$ that $v_\tau(t)\to v(t)$ weakly in $H$.
  From hypothesis $H_1(2)$ and the uniqueness of the weak limit, we have
\begin{equation}\label{KB_13}
  v(0)=v_0\ \textrm{and}\ v^0 \to v_0\ \textrm{in}\ H.
\end{equation}
  The continuity of the norm implies $|v_\tau(0)|\to |v(0)|$.
  Moreover, $v_\tau(T)\to v(T)$ weakly in $H$ and by the weak lower semicontinuity of norm, we have
  $|v(T)|\leqslant \liminf|v_\tau(T)|$. Summarizing, we conclude that
\begin{equation}
  \limsup (\hat{v}_\tau ',v-v_\tau )_\mathcal{H} \leqslant 0. \label{26}
\end{equation}
  Since the operator $B_0\colon V\to V^*$ defines the inner product on $V$ and since $(Kw)'=w$ for all $w\in L^2(0,T;V)$,
  integrating by parts, we get
\begin{eqnarray}\label{27}
  & & \left\langle\mathcal{B}_0 Kw,w\right\rangle_{\mathcal{V}^*\times\mathcal{V}}\cr
  & = & \int_0^T\left\langle \mathcal{B}_0(Kw)(t),w(t)\right\rangle_{V^*\times V} \,dt \cr
  & = & \int_0^T \left\langle \mathcal{B}_0(Kw)(t),(Kw)'(t)\right\rangle_{V^*\times V} \,dt\cr
  & = & \frac{1}{2}\langle \mathcal{B}_0(Kw)(T),(Kw)(T)\rangle_{V^*\times V} - \frac{1}{2}\langle \mathcal{B}_0(Kw)(0),(Kw)(0)\rangle_{W^*\times W}\cr
  & = & \frac{1}{2}\|Kw(T)\|_B^2 \geqslant 0.
\end{eqnarray}
  By (\ref{KB_9}) and (\ref{27}) we have
\begin{eqnarray}
  \nonumber
 & & \dual{B_0u_\tau,v-v_\tau}{\mathcal{V}}=\dual{B_0u,v-v_\tau}{\mathcal{V}}\\
 \nonumber
 & & +\dual{B_0(u_\tau-u_0-Kv_\tau),v-v_\tau}{\mathcal{V}}-\dual{B_0K(v-v_\tau),v-v_\tau}{\mathcal{V}}\\
 \label{KB_15}
 & & \leqslant\dual{B_0u,v-v_\tau}{\mathcal{V}}+\dual{B_0(u_\tau-u_0-Kv_\tau),v-v_\tau}{\mathcal{V}}.
\end{eqnarray}
  From Lemma \ref{convergences}$(b)$, it follows that $v_\tau\to v$ weakly in $\mathcal{V}$. Thus
\[
  \dual{B_0u,v-v_\tau}{\mathcal{V}}\to 0.
\]
  Moreover, by Lemma
  \ref{convergences}$(f)$, we also have
\[
  \dual{B_0(u_\tau-u_0-Kv_\tau),v-v_\tau}{\mathcal{V}}\to 0.
\]
  Thus, it follows from (\ref{KB_15}), that
\begin{equation}\label{28}
  \limsup\langle \mathcal{B}_0 u_\tau,v-v_\tau\rangle_{\mathcal{W}^*\times\mathcal{W}} \leqslant 0.
\end{equation}
From \eqref{19} and Lemma~\ref{convergences}$(b)$, we obtain
\begin{equation}\label {29}
  \lim\langle \mathcal{C}_\tau u_\tau,v_\tau-v\rangle_{\mathcal{W}^*\times\mathcal{W}} = 0.
\end{equation}
  From Lemma \ref{convergences} $(g)$, $(l)$ and hypothesis $H(\gamma)$, passing to a subsequence if necessary, we have
\begin{equation}\label{30}
  \langle\eta_\tau,\bar{\gamma}v_\tau -\bar{\gamma} v\rangle_{\mathcal{U}^*\times \mathcal{U}} \to 0.
\end{equation}
  Applying \eqref{26}-\eqref{30} in \eqref{25} we have
\begin{equation}\label{31}
  \limsup\langle \mathcal{A}_\tau v_\tau,v_\tau-v\rangle_{\mathcal{W}^*\times\mathcal{W}} \leqslant 0.
\end{equation}
  From \eqref{23}, \eqref{24}, \eqref{31} in \eqref{222}, we get
\begin{equation}\label{32}
  \limsup\langle\mathcal{A}v_\tau,v_\tau -v\rangle_{\mathcal{W}^*\times\mathcal{W}} \leqslant 0.
\end{equation}
  Now, from Lemma~\ref{convergences}$(b)$, $(l)$, \eqref{32} and Lemma~\ref{kalita} we obtain
\begin{equation}\label{33}
  \mathcal{A}v_\tau \to \mathcal{A}v \quad \textrm{weakly in}\ \mathcal{W}^*.
\end{equation}
  By (\ref{33}) and (\ref{KB_12}) we obtain (\ref{KB_11}).
  Using \eqref{16}-\eqref{KB_11} we pass to the limit in \eqref{Nem1} and obtain
\begin{equation}\label{KB_14}
  v' + \mathcal{A}v + \mathcal{B}_0 u + \mathcal{C}u + \bar{\gamma}^*\eta = f.
\end{equation}
  Next, we pass to the limit with inclusion \eqref{Nem2}. From Lemma~\ref{convergences}$(l)$
  and hypothesis $H(\bar{\gamma})$, we have that $\bar{\gamma} v_\tau \to\bar{\gamma} v$ in $\mathcal{U}$
  and in consequence
\begin{equation}
  \bar{\gamma}v_\tau (t) \to \bar{\gamma} v(t) \quad \textrm{in}\ U, \quad \textrm{for a.e.}\ t\in[0,T]. \label{35}
\end{equation}
  From Lemma \ref{convergences} $(g)$, \eqref{35} and Aubin-Celina convergence theorem (cf. \cite{bAubinCelina1984a}), we get
\begin{equation}\label{KB_16}
  \eta(t)\in M(\bar{\gamma} v(t)) \quad \textrm{for a.e.} \quad t\in[0,T].
\end{equation}
  Moreover, by Lemma~\ref{convergences} we have $u=u_0+Kv$, thus $u(0)=u_0$ and $u'=v$. Combining it with \eqref{KB_13}, we see that $u'(0)=v_0$. This, together
  with \eqref{KB_14} and \eqref{KB_16} shows that $u$ is a solution of Problem {$\mathcal{P}$}. This completes the proof.
\end{proof}

\section{Examples}\label{Examples}
In this section we consider two problems, for which the existence result obtained in Theorem \ref{theorem_1} is applicable.

 Let $\Omega$ be an open bounded
 subset of $\mathbb{R}^N$ with a Lipschitz boundary $\partial\Omega$.
 The boundary is divided in two parts $\Gamma_1$, $\Gamma_2$ such that
$\bar{\Gamma_1}\cup\bar{\Gamma_2}={\partial\Omega}$ and the $N-1$ dimensional measure of $\Gamma_1$ is positive. We denote by $\nu$ the  outward unit vector
normal to $\partial\Omega$. Let $p\geq 2$, $T>0$ and let the functions $g\colon \real\to\real$, $j_1\colon \Gamma_2\to\real$, $j_2\colon \Omega\to\real$,
$f_1\colon [0,T]\times\Omega\to\real$ and $f_2\colon [0,T]\times\Omega\to\real$ be given. We formulate two problems.\\

\noindent {\bf Problem $P_1$}. Find $u\colon [0,T]\times\Omega\to\real$ such that
    \begin{eqnarray*}
      \left\{
      \begin{array}{l}
        u''-\alpha_1 {\rm div}\left(|\nabla u'|^{p-2}\nabla u'\right) + g(u')-\Delta u+ |u|^\delta u=f_1 \quad \textrm{in}\ \Omega\times (0,T),\\
        u=0 \quad \textrm{on}\ \Gamma_1,\\
        \frac{\partial u}{\partial \nu} + \nu\cdot\left(|\nabla u'|^{p-2} \nabla u'\right)= \eta \in \partial j_1(\gamma u') \quad \textrm{on}\ \Gamma_2\times (0,T),\\
        u(0)=u_0, \qquad u'(0)=v_0.
      \end{array}
      \right.
    \end{eqnarray*}

\noindent {\bf Problem $P_2$}. Find $u\colon [0,T]\times\Omega\to\real$ such that
    \begin{eqnarray*}
      \left\{
      \begin{array}{l}
        u''-\alpha_2 {\rm div}\left(|\nabla u'|^{p-2}\nabla u'\right) + g(u')-\Delta u+ |u|^\delta u+\gamma^*\eta=f_2 \quad \textrm{in}\ \Omega\times (0,T),\\
        \eta\in\partial j_2(u')\quad \textrm{in}\ \Omega\times (0,T),\\
        u=0 \quad \textrm{on}\ \partial\Omega\times (0,T),\\
        u(0)=u_0, \qquad u'(0)=v_0.
      \end{array}
      \right.
    \end{eqnarray*}

\noindent In the above problems $\partial j_i$ denotes the Clarke subdifferential of the function $j_i$, $i=1,2$, $\alpha_i >0$ are constants and
$\delta\leqslant 1-\frac{2}{p}$.

We impose the following assumptions on the functions $g$, $j_1$ and $j_2$.\\

\noindent ${\underline{H(g)}}$ $g\colon \real\to\real$ is such that
    \begin{itemize}
        \item[(i)] $g$ is continuous,
        \item[(ii)] $\inf\limits_{s\in\real}g(s)s>-\infty$,
        \item[(iii)] $|g(s)|\leqslant c_g(1+|s|^{p-1})$ for all $s\in\real$ with $c_g>0$.
    \end{itemize}

\medskip

\noindent $\underline{H(j_1)}$ $j_1\colon \Gamma_2\times\real\to\real$ is such that
\begin{itemize}
    \item[(i)] $j_1(\cdot,\xi)$ is measurable for all $\xi\in\real$ and $j_1(\cdot,0)\in L^1(\Gamma_2)$,
    \item[(ii)] $j_1(x,\cdot)$ is locally Lipschitz for a.e. $x\in\Gamma_2$,
    \item[(iii)] $|\eta|\leqslant c_{j_1}(1+|\xi|^{p-1})$ for all $\eta\in\partial j_1(x,\xi)$,  $x\in\Gamma_2$ with $c_{j1}>0.$
\end{itemize}

\noindent $\underline{H(j_2)}$ $j_2\colon \Omega\times\real\to\real$ is such that
\begin{itemize}
    \item[(i)] $j_2(\cdot,\xi)$ is measurable for all $\xi\in\real$ and $j_2(\cdot,0)\in L^1(\Omega)$,
    \item[(ii)] $j_2(x,\cdot)$ is locally Lipschitz for a.e. $x\in\Omega$,
    \item[(iii)] $|\eta|\leqslant c_{j_2}(1+|\xi|^{p-1})$ for all $\eta\in\partial j_2(x,\xi)$,  $x\in\Omega$ with $c_{j2}>0.$
\end{itemize}

\noindent We introduce the spaces $W_1=\{v\in W^{1,p}(\Omega),\,\,v=0\,\,\text{on}\,\,\Gamma_1\}$ and  $W_2=W_0^{1,p}(\Omega)$ equipped with the norm
\begin{align}
\|v\|_{W_1}=\|v\|_{W_2}=\left(\int_{\Omega}|\nabla v(x)|^p\,dx\right)^{\frac{1}{p}}.\nonumber
\end{align}

\noindent Moreover, we define spaces $V_1=\{v\in H^1(\Omega):\ v=0\,\,\text{on}\,\,\Gamma_1\}$,  $V_2=H_0^1(\Omega)$, $H=L^2(\Omega)$, equipped with the norm
\begin{align}
\|v\|_{V_1}=\|v\|_{V_2}=\left(\int_{\Omega}|\nabla v(x)|^2\,dx\right)^{\frac{1}{2}}.\nonumber
\end{align}

\noindent Finally, we take $U_1=L^p(\Gamma_1)$ and $U_2=L^p(\Omega)$. Next, we consider operators $A_1\colon W_1\to W_1^*$, $A_2\colon W_2\to W_2^*$,
$B_1\colon V_1\to W_1^*$ and $B_2\colon V_2\to W_2^*$  defined by
\begin{eqnarray*}
  \langle A_1u,v\rangle_{W_1^*\times W_1} & = & \alpha_1\int_{\Omega} |\nabla u|^{p-2} \nabla u \cdot\nabla v\,dx + \int_\Omega g(u)\cdot v\,dx\ \text{for all}\,\,\,u, v\in W_1,\\
  \langle A_2u,v\rangle_{W_2^*\times W_2} & = & \alpha_2\int_{\Omega} |\nabla u|^{p-2} \nabla u \cdot\nabla v\,dx + \int_\Omega g(u)\cdot v\,dx\ \text{for all}\,\,\,u, v\in W_2,\\
  \langle B_1u,v\rangle_{W^*_1\times W_1} & = & \int_\Omega \nabla u\cdot\nabla v dx + \int_\Omega|u|^\delta u\cdot v\,dx\,\,\,\text{for all}\,\,\,u\in V_1, v\in W_1,\\
  \langle B_2u,v\rangle_{W^*_2\times W_2} & = & \int_\Omega \nabla u\cdot\nabla v dx+ \int_\Omega|u|^\delta u\cdot v\,dx\,\,\,\text{for all}\,\,\,u\in V_2, v\in W_2.
\end{eqnarray*}

  We define the spaces $\mathcal{W}_1=L^p(0,T;W_1)$, $\mathcal{W}_2=L^p(0,T;W_2)$, $\mathcal{V}=L^p(0,T;V)$, ${\mathcal{H}}=L^2(0,T;H)$,
 $\mathcal{U}_1=L^p(0,T;U_1)$ and $\mathcal{U}_2=L^p(0,T;U_2)$.

  We need the following assumptions on the right hand side of Problems $P_1$ and $P_2$.\\

\noindent $\underline{H(f_1)}$: $f_1\in {\mathcal{W}}_1$.\\

\noindent $\underline{H(f_2)}$: $f_1\in {\mathcal{W}}_2$.\\

 We define the functionals $F_1\in W_1^* $ and $F_2\in W_2^*$ by
\[
  F_1(v)=\int_{\Omega}f_1\cdot v\,dx\ \text{for all}\ v\in W_1,\quad F_2(v)=\int_{\Omega}f_2\cdot v\,dx\ \text{for all}\ v\in W_2.
\]

 Now we introduce the notion of a weak solution of Problems ${P}_1$ and ${P}_2$.

\begin{Definition}\label{def_weak_1}
A function $u\in \mathcal{W}_1$ is said to be a weak solution of Problem~${P_1}$ if $u'\in\mathcal{W}_1$, $u''\in\mathcal{W}_1^*$ and satisfies
\[
  \left\{
  \begin{array}{l}
  \dual{u''(t)+A_1u'(t)+B_1u(t),v}{W_1}+\int_{\Gamma_2}\eta(x)v(x)\,d\Gamma=F_1(v)\\
  \hspace*{5cm}\text{for a.e.}\ t\in (0,T),\,\,\,\text{for all}\,\,\, v\in W_1,\\
  \eta(x)\in \partial j_1(u'(x))\,\,\,\text{for a.e.}\ x\in\Gamma_2,\\
   u(0)=u_0, \qquad u'(0)=v_0.
   \end{array}
   \right.
\]
\end{Definition}

\begin{Definition}\label{def_weak_2}
A function $u\in \mathcal{W}_2$ is said to be a weak solution of Problem~${P_2}$ if $u'\in\mathcal{W}_2$, $u''\in\mathcal{W}_2^*$ and satisfies
\[
  \left\{
  \begin{array}{l}
    \dual{u''(t)+A_2u'(t)+B_2u(t),v}{W_2}+\int_{\Omega}\eta(x)v(x)\,dx=F_2(v)\\
    \qquad\qquad\qquad\qquad\qquad\qquad\quad\text{for a.e.}\ t\in (0,T),\,\,\,\text{for all}\,\,\, v\in W_2,\\
    \eta(x)\in \partial j_2(u'(x))\,\,\,\text{for a.e.}\ x\in\Omega,\\
    u(0)=u_0, \qquad u'(0)=v_0.
   \end{array}
   \right.
\]
\end{Definition}

  We remark that the weak formulations in Definitions \ref{def_weak_1} and \ref{def_weak_2}
  are obtained from equations in Problems $P_1$ and $P_2$,
  respectively, by multiplying them by a test function $v\in W_1$ ($v\in W_2$, respectively)  and using the Green formula.

 In what follows we will deal with the existence of weak solutions of Problems~$P_1$ and $P_2$. First we define two auxiliary functionals
  $J_1\colon U_1\to\real$ and $J_2\colon U_2\to\real$ given by
\[
  J_1(v)=\int_{\Gamma_2}j_1(x,v(x))\, d\Gamma\ \textrm{for all}\ v\in U_1,
\]
\[
  J_2(v)=\int_{\Omega}j_2(x,v(x))\,dx\ \textrm{for all}\ v\in U_2.
\]

  Next, we define the multifunctions $M_1\colon U_1\to2^{U_1^*}$
  and $M_2\colon  U_2\to2^{U_2^*}$ given by $M_i(v)=\partial J_i(v)$
  for all $v\in U_i$, $i=1,2$. Finally, let
  $\gamma_1\colon W_1\to U_1$ denote the trace operator
  and $\gamma_2\colon W_2\to U_2$ the embedding operator. Now, we formulate two auxiliary problems.\\

\noindent {\bf Problem {${\mathcal{P}}_1$}}. Find $u\in \mathcal{W}_1$ with $u'\in\mathcal{W}_1$ and $u''\in\mathcal{W}_1^*$ such that
\begin{eqnarray}
 && \!\!\!\!\!\!\!\!\!\!  u''(t) + A_1(u'(t)) + B_1(u(t)) + \gamma_1^* M_1(\gamma_1 u'(t)) \ni f_1(t)\ \mbox{a.e.}\ {t\in [0,T]},\nonumber\\
 && \!\!\!\!\!\!\!\!\!\!  u(0)=u_0, \quad
u'(0)=v_0.\nonumber
\end{eqnarray}

\noindent {\bf Problem {${\mathcal{P}}_2$}} Find $u\in \mathcal{W}_2$ with $u'\in\mathcal{W}_2$ and $u''\in\mathcal{W}_2^*$ such that
\begin{eqnarray}
 && \!\!\!\!\!\!\!\!\!\!  u''(t) + A_2(u'(t)) + B_2(u(t)) + \gamma_2^* M_2(\gamma_2 u'(t)) \ni f_2(t)\ \mbox{a.e.}\ {t\in [0,T]},\nonumber\\
 && \!\!\!\!\!\!\!\!\!\!  u(0)=u_0, \quad
u'(0)=v_0.\nonumber
\end{eqnarray}

\begin{Remark}\label{Remark_1}
  By the properties of Clarke subdifferential of functionals $J_i$,
  it follows that each solution of Problem {${\mathcal{P}}_i$} is also
  a solution of Problem $P_i$, $i=1,2$.
\end{Remark}

\noindent We recall that the following Poincare inequalities hold
\begin{align}
&\int_{\Omega}|v(x)|^p dx\leq \tilde{c}_1\int_{\Omega}|\nabla v(x)|^p dx\,\,\,\text{for all}\ v\in W_1,\label{KB_19}\\
&\int_{\Omega}|v(x)|^p dx\leq \tilde{c}_2\int_{\Omega}|\nabla v(x)|^p dx\,\,\,\text{for all}\ v\in W_2\label{KB_20}
\end{align}
with $\tilde{c}_1, \tilde{c}_2>0$.
Let us define the following constants
\begin{align*}
& c_{A_1}=\max\big\{c_q\tilde{c}_1^{\frac{1}{p}}|\Omega|^{\frac{1}{q}},\alpha_1+c_g\tilde{c}_1^{\frac{1}{pq}}\big\},\\
& c_{A_2}=\max\big\{c_q\tilde{c}_2^{\frac{1}{p}}|\Omega|^{\frac{1}{q}},\alpha_1+c_g\tilde{c}_2^{\frac{1}{pq}}\big\},\\
& c_{M_1}=c_{j_1}2^{\frac{1}{p}}\max\big\{1,|\Gamma_2|^{\frac{1}{q}}\big\},\\
& c_{M_2}=c_{j_2}2^{\frac{1}{p}}\max\big\{1,|\Omega|^{\frac{1}{q}}\big\},
\end{align*}
where $|\Gamma_2|$ and $|\Omega|$ denote the surface measure of $\Gamma_2$ and the Lebesgue measure of $\Omega$, respectively. Now we formulate lemmata
containing the properties of the operators $A_1$, $A_2$, $M_1$ and $M_2$.

\begin{Lemma}\label{lemma_A1}
  If assumption $H(g)$ holds, then operator $A_1$ satisfies
\begin{itemize}
  \item[(i)] $\|A_1u\|_{W_1^*}\leq c_{A_1}(1+\|u\|_{W_1}^{p-1})$ \,\,\,for all $u\in W_1$,
  \item[(ii)] $\dual{A_1u,u}{W_1}\geq\alpha_1\|u\|_{W_1}^p+\inf_{s\in\real}g(s) s|\Omega|$ for all $u\in W_1$,
  \item[(iii)] $A_1$ is pseudomonotone.
\end{itemize}
\end{Lemma}

\begin{proof}
 Condition $(i)$ follows from $H(g)(iii)$ and \eqref{KB_19}. Condition $(ii)$
 follows directly from the definition of $A_1$ and $H(g)(ii)$.
 Finally, for the pseudomonotonicity of $A_1$, we refer to Chapter 2 of
 \cite{Roubicek2005}.
\end{proof}

\begin{Lemma}\label{lemma_A2}
  If assumption $H(g)$ holds, then operator $A_2$ satisfies
\begin{itemize}
\item[(i)] $\|A_2u\|_{W_2^*}\leq c_{A_2}(1+\|u\|_{W_2}^{p-1})$ \,\,\,for all $u\in W_2$,
\item[(ii)] $\dual{A_2u,u}{W_2}\geq\alpha_2\|u\|_{W_2}^p+\inf_{s\in\real}g(s) s|\Omega|$\,\,\, for all $u\in W_2$,
\item[(iii)] $A_2$ is pseudomonotone.
\end{itemize}
\end{Lemma}
\noindent The proof of Lemma \ref{lemma_A2} is analogous to the proof of Lemma \ref{lemma_A1}.

\begin{Lemma}\label{lemma_M1}
  If assumption $H(j_1)$ holds, then operator $M_1$ satisfies
\begin{itemize}
 \item[(i)] for all $u\in U_1$, $M_1(u)$ is a nonempty, closed and convex set,
 \item[(ii)] $M_1$ is upper semicontinuous in $(s\textrm{-}U_1\times w\textrm{-}U_1^*)$-topology,
 \item[(iii)] $\|\eta\|_{U_1^*} \leqslant c_{M_1}(1+\|w\|_{U_1}^{p-1})$ for all $w\in U_1$, all $\eta \in M_1(w)$.
\end{itemize}
\end{Lemma}

\begin{Lemma}\label{lemma_M2}
  If assumption $H(j_2)$ holds, then operator $M_2$ satisfies
\begin{itemize}
 \item[(i)] for all $u\in U_2$, $M_2(u)$ is a nonempty, closed and convex set,
 \item[(ii)] $M_2$ is upper semicontinuous in $(s\textrm{-}U_2\times w\textrm{-}U_2^*)$-topology,
 \item[(iii)] $\|\eta\|_{U_2^*} \leqslant c_{M_2}(1+\|w\|_{U_2}^{p-1})$ for all $w\in U_2$, all $\eta \in M_2(w)$.
\end{itemize}
\end{Lemma}

 Let $\bar{\gamma}_i\colon {\mathcal{W}}_i\to{\mathcal{U}}_i$ be Nemytskii operator corresponding to $\gamma_i$ defined by $(\bar{\gamma}_iv)(t)=\gamma_iv(t)$
 for all $v\in {\mathcal{W}}_i$, $i=1,2$. The following lemmata deal with the properties of $\bar{\gamma}_1$ and $\bar{\gamma}_2$.

\begin{Lemma}\label{lemma_gamma1}
  The Nemytskii operator $\bar{\gamma}_1\colon M^{p,q}(0,T;W_1,W_1^*)\to L^q(0,T;U_1^*)$ is compact.
\end{Lemma}

\begin{proof}
  Let $\varepsilon\in (0,\frac{1}{2})$. Then the embedding $i\colon W_1\to
  W^{1-\varepsilon,p}(\Omega)$ is compact.
  The trace operator $\tilde{\gamma}_1\colon W^{1-\varepsilon,p}(\Omega)\to W^{\frac{1}{2}-\varepsilon,p}(\partial\Omega)$
  is linear and continuous and, finally, the embedding
  $j\colon W^{\frac{1}{2}-\varepsilon,p}(\partial\Omega)\to L^p(\partial\Omega)=U_1$ is also linear and continuous.
  Thus $\gamma_1=j\circ\tilde{\gamma}_1\circ i$ is linear, continuous and compact. Moreover, the spaces
  $V_1\subset W^{1-\varepsilon,p}(\Omega)\subset V_1^*$ satisfy assumptions of Proposition \ref{kalita_compactness} so the embedding
  $M^{p,q}(0,T;V_1,V_1^*)\subset L^p(0,T;W^{1-\varepsilon,p}(\Omega))$ is compact.
  Since the embedding $L^p(0,T;W^{1-\varepsilon,p}(\Omega))\subset \mathcal{U}_1$
  is continuous the Nemytskii operator corresponding to $\gamma_1$ is compact.
\end{proof}

\begin{Lemma}\label{lemma_gamma2}
  The Nemytskii operator $\bar{\gamma}_2\colon M^{p,q}(0,T;W_2,W_2^*)\to L^q(0,T;U_2^*)$ is compact.
\end{Lemma}

\begin{proof}
  Use directly Proposition \ref{kalita_compactness} to the triple of spaces $W_2$, $U_2^*$ and $W_2^*$.
\end{proof}

  Now we impose additional assumptions on the constants of the problems.\\

\noindent $\underline{H_0^1}$: $\alpha_1>c_{M_1}\|\gamma_1\|^p_{\mathcal{L}(W_1,U_1)}$,\\

\noindent $\underline{H_0^2}$: $\alpha_2>c_{M_2}\|\gamma_2\|^p_{\mathcal{L}(W_2,U_2)}$.\\

   We are in a position to formulate the existence results for Problems~$P_1$ and~$P_2$.
\begin{Theorem}
  Let assumptions $H(g)$, $H(j_1)$, $H(f_1)$, $H_0^1$ ($H(j_2)$, $H(f_2)$, $H_0^2$, respectively) hold and $u_0\in V, v_0\in H$.
  Then Problem $P_1$ (Problem  $P_2$,
  respectively) admits a weak solution.
\end{Theorem}
\begin{proof}
We apply Theorem \ref{theorem_1} to Problems {${{\mathcal{P}}}_1$} and {${{\mathcal{P}}}_2$}. To this end we observe that Lemmata \ref{lemma_A1} and
\ref{lemma_A2} imply that operators $A_1$ and $A_2$ satisfy assumptions corresponding to $H(A)$. It is also clear that both operators $B_1$ and $B_2$ can be
represented as a sum of linear term $B_0$ and nonlinear one $C$, which satisfy assumptions corresponding to $H(B_0)$ and $H(C)$ (see example in
\cite{Emrichh_2010}). Moreover, Lemmata \ref{lemma_M1} and  \ref{lemma_M2} provide that the multivalued operators $M_1$ and $M_2$ satisfy assumptions analogous
to $H(M)$. Similarly, Lemmata \ref{lemma_gamma1} and \ref{lemma_gamma2} guaranty that assumption $H(\gamma)$ is fulfilled in case of operators $\gamma_1$ and
$\gamma_2$. Finally, assumptions $H_0^1$ and $H_0^2$ are analogous to assumption $H_0$ of Theorem \ref{theorem_1}. Including assumptions $H(f_1)$ and $H(f_2)$,
we are in a position to use Theorem \ref{theorem_1} and obtain the existence of solution to Problems {${{\mathcal{P}}}_1$} and {${{\mathcal{P}}}_2$}. From
Remark \ref{Remark_1} and Definitions \ref{def_weak_1} and \ref{def_weak_2}, we get that Problems $P_1$ and $P_2$ admit weak solutions.
\end{proof}


\begin{thebibliography}{99}



\bibitem{bAubinCelina1984a}
  J.P. Aubin, H. Frankowska,
  Set-Valued Analysis, Birkh\"auser,
  Boston, Basel, Berlin (1990).


\bibitem{HMSBOOK} K. Bartosz, Numerical Methods for Evolution Hemivariational
Inequalities, Chapter 5 in W. Han, S. Mig\'orski, M. Sofonea, {\em Advances in Variational and
    Hemivariational Inequalities. Theory, Numerical Analysis and
    Applications}, Advances in Mechanics and Mathematics {\bf 33},
Springer, New York, 2015.

\bibitem{Bartosz_theta}
  K. Bartosz,
  \emph{Variable time-step $\theta$-scheme for nonautonomous second order evolution inclusion},
  submitted to Applied Mathematical Modelling.


\bibitem{bxkyz}
  K. Bartosz, X. Cheng, P. Kalita, Y. Yu, C. Zheng,
  \emph{Rothe method for evolution variational-hemivariational inequalities},
  J. Math. Anal. Appl.,
  \textbf{423} (2015), 841--862.


\bibitem{Bartosz_Sofonea}
  K. Bartosz, M. Sofonea,
  \emph{The Rothe method for variational-hemivariational inequalities with applications to contact mechanics},
  submitted to SIAM Journal on Mathematical Analysis.


\bibitem{Emrichh_2010}
  E. Emrich, M. Thalhammer,
  \emph{Convergence of a time discretisation for doubly nonlinear evolution equations of second order},
  Found. Comput. Math.,
  \textbf{10} (2010), 171--190, DOI 10.1007/s10208-010-9061-5.

\bibitem{Kalita2013}
  P. Kalita,
  \emph{Convergence of Rothe scheme for hemivariational inequalities of parabolic type},
  Int. J. Numer. Anal. Mod.,
  \textbf{10} (2013) 445--465.

\bibitem{Kalita2}
  P. Kalita,
  \emph{Regularity and Rothe method error estimates for parabolic hemivariational inequality},
  J. Math. Anal. Appl.,
  \textbf{389} (2012), 618--631.


\bibitem{Kalita3}
  P. Kalita,
  \emph{Semidiscrete variable time-step $\theta$-scheme for nonmonotone evolution inclusion},
  arXiv:1402.3721.

\bibitem{JeongPark}
  J. M. Jeong, J. Y. Park, and S. H. Park,
  \emph{Hyperbolic hemivariational inequalities with boundary source and damping terms},
  Commun. Korean Math. Soc.,
  \textbf{24} (2009), 1, 85--97.

\bibitem{Peng1}
  Z. Peng, Z. Liu,
  \emph{Evolution hemivariational inequality problems with doubly nonlinear operators},
  J. Glob. Optim.,
  \textbf{51} (2011), 413--427.

\bibitem{Peng2}
  Z. Peng, Z. Liu, X. Liu,
  \emph{Boundary hemivariational inequality problems with doubly nonlinear operators},
  Math. Ann.,
  \textbf{356} (2013), 1339--1358.

\bibitem{Peng3}
  Z. Peng, C. Xiao,
  \emph{Existence and convergence theorems for evolutionary hemivariational inequalities of second order},
  Electronic Journal of Differential Equations,
  \textbf{65} (2015), 1--17.


\bibitem{Roubicek2005}
  T. Roubi\v{c}ek,
  Nonlinear Partial Differential Equations with Applications,
  Birkh\"{a}user Verlag, Basel, Boston, Berlin, 2005.

\bibitem{Z}
  E. Zeidler,
  Nonlinear Functional Analysis and Applications II A/B,
  Springer, New York, 1990.
\end{thebibliography}
\end{document}